\documentclass[11pt,letterpaper]{amsart}
\usepackage{amssymb, amscd, amsfonts, euler, epic, xy}
\usepackage{srcltx}

\usepackage[
            pdftex,
           colorlinks=true,
            urlcolor=blue,       
            filecolor=blue,     
            linkcolor=blue,       
            citecolor=blue,         
            pdftitle= {Title},
            pdfauthor={Author},
            pdfsubject={Subject},
            pdfkeywords={Key}
            pagebackref,
            pdfpagemode=None,
            bookmarksopen=true]
            {hyperref}
\usepackage{hyperref}
\usepackage{color}

\def\CC{{\mathbb C}}

\def\NN{{\mathbb N}}

\def\QQ{{\mathbb Q}} 
\def\RR{{\mathbb R}} 
\def\ZZ{{\mathbb Z}}

\def\v{{\scriptscriptstyle{\vee}}}
\def\d{\dagger}

\def\lrp{locally rationally polyhedral}
\def\rp{rationally polyhedral}
\def\rp{rationally polyhedral}
\def\nbhd{neighborhood}
\def\half{\tfrac{1}{2}}
\def\G{{\Gamma}}
\def\g{{\gamma}}

\def\pt{{\bullet}}

\def\Gcal{{\mathcal G}} 
\def\Hcal{{\mathcal H}}

\def\la{\langle}
\def\ra{\rangle}

\newcommand\aff{\operatorname{Aff}}
\newcommand\ann{\operatorname{Ann}}
\newcommand\aut{\operatorname{Aut}}  
\newcommand\codim{\operatorname{codim}} 
 
\newcommand\Hom{\operatorname{Hom}}

\newcommand\sym{\operatorname{Sym}}
\newcommand\GL{\operatorname{GL}}

\newcommand\SL{\operatorname{SL}}

\newcommand\Star{\operatorname{Star}}
\newcommand\Orth{\operatorname{O}}

\newcommand\PSL{\operatorname{PSL}}
\newcommand\Tr{\operatorname{T}}
\newcommand\as{\operatorname{As}}

\newtheorem{theorem}{Theorem}[section]
\newtheorem{lemma}[theorem]{Lemma}
\newtheorem{proposition}[theorem]{Proposition}
\newtheorem{propdef}[theorem]{Proposition-Definition}

\newtheorem{corollary}[theorem]{Corollary}

\theoremstyle{definition}
\newtheorem{definition}[theorem]{Definition}
\newtheorem{application}[theorem]{Application}

\newtheorem{example}[theorem]{Example}

\theoremstyle{remark} 
\newtheorem{remark}[theorem]{Remark}

\newtheorem{question}[theorem]{Question}

\begin{document}

\title[Convex cones of finite type]{Discrete automorphism groups of convex cones of finite type}
\author{Eduard Looijenga}
\email{Eduard@math.tsinghua.edu.cn}
\address{Mathematical Sciences Center, Jin Chun Yuan West Building\\ 
Tsinghua University\\
Haidan District, Beijing 100084, P.R. China}

\begin{abstract}
We investigate subgroups of $\SL (n,\ZZ)$ which preserve an open nondegenerate convex cone in $\RR^n$ and admit in that cone as fundamental domain a polyhedral cone of which some faces are allowed to  lie on the boundary. Examples are arithmetic groups acting on self-dual cones, Weyl groups of certain Kac-Moody algebras and do occur in algebraic geometry as the
automorphism groups of projective manifolds acting on their ample cones. 
\end{abstract}

\maketitle

\section*{Introduction}
This story begins with the seemingly innocuous Theorem \ref{thm:bdlypol}, which might have a place in  the theory of linear programming. It is easily stated, even in an introduction: if $V$  is a finite dimensional real vector space, $L\subset V$ a lattice  and $C$ an open nondegenerate convex cone in $V$, then the convex hull of $C\cap L$ is locally polyhedral in the sense that its intersection with any bounded polyhedron is a polyhedron. This is our basic tool for our investigation of the linear automorphism groups $\G$ of  $C$ which preserve $L$ and possess a natural finiteness property. That property has a number of equivalent formulations, one of which is that there exists a  convex cone spanned by a finite subset of $L\cap \bar C$ whose $\G$-orbit  contains $C$.  This turns out to be of a self-dual nature in the sense that the same is then true for the contragradient action of $\G$ on $V^*$ relative to the  duals of $C$ and $L$. Since the invariant lattice  $L$ is secondary to the resulting $\QQ$-structure $V(\QQ)$ on $V$, we call 
$(V(\QQ),C,\G)$ a \emph{polyhedral triple}.

Examples of polyhedral triples abound and  provide sufficient justification for investigating this situation in its own right, even if  the motivation lies elsewhere (more on this  at the end of this introduction). First of all there is the case  when $C$ is a self-dual homogeneous cone and $\G$ is an arithmetic subgroup of the automorphism group of $C$ (the most classical instances of which are perhaps the Lobatchevski cones and the  cone of positive definite quadratic forms in a fixed number of variables). This is the case studied in detail by Avner Ash in \cite{amrt} and indeed, much of the present paper generalizes his work to our setting.

Another class of examples constitute the cones attached to irreducible Coxeter groups  that are neither finite nor of affine type (these are indeed nondegenerate convex). Related to this class of are the hyperbolic reflection groups studied by Vinberg and Nikulin.
A number of examples occur in algebraic geometry by taking for $V$ the N\'eron-Severi space
of a projective manifold, for $C$ its ample cone (or dually, the cone spanned be the homology classes of curves) and for $\G$ the image of the representation of  automorphism group of the manifold. 

The first section recalls (or discusses) some of the basics of the theory of convex sets.
In Section 2 we prove the fundamental theorem stated above. Its first applications are in 
Section 3. The following section introduces the central notion of this paper 
(that of a polyhedral triple) and we derive a number of properties  of the corresponding group.
In Section 5  we determine the structure of a stabilizer of a face, with most of our results being  summed up  by  Theorem \ref{thm:stabilizerfacemain}.
\\

This paper intends to be the first installment  of a series on semi-toric compactification.
It is a `spin-off' of my ancient unpublished preprint \cite{looij:semitoric} to the extent
that it  is only about the geometry of discrete groups acting  on convex cones, 
the complex-analytic story being relegated to sequels. Yet I  believe that  
this material forms a natural whole and that the results have an interest of their own, independent of the motivating application originally envisaged (namely to develop a common generalization of the compactifications of Baily-Borel and Mumford \emph{et al.}\  \cite{amrt} of locally symmetric varieties).
\\

\emph{Acknowledgements.}
I am indebted to Burt Totaro for a meticulous reading of an earlier version of this paper. He pointed out several inaccuracies and often 
suggested ways to overcome them.  His comments also helped to make this version more reader friendly. (In the mean time
he has applied some of the results presented here \cite{totaro}.) I thank Ofer Gabber for a suggestion made long ago for a proof of Lemma \ref{lemma:unip}.
\\

\emph{A notational convention.} If a group $\G$ acts on a set $X$ and $Y\subset X$, then we denote by $N_\G (Y)$ resp.\  $Z_\G (Y)$ the subgroup of $\g\in \G$ which leave $Y$ invariant resp.\ pointwise fixed so that $\G (Y):=N_\G (Y)/Z_\G (Y)$ can be understood as a group of permutations of $Y$.

\section{Convex Cones and Kernels}\label{section:convcones}

We begin with recalling some definitions that are 
standard in the theory of convex sets.
A subset $C$ of a real finite dimensional vector space $V$
is called a \emph{cone} if it is nonempty and
invariant under scalar multiplication with positive
numbers (so $C$ need not contain the origin). The set of
linear forms on $V$ that are $\geq 0$ on $C$ is 
a closed convex cone in the dual vector space $V^*$,  called 
the \emph{cone dual to $C$} and denoted $C^*$. 
The interior of this dual, to which we 
shall refer as the \emph{open dual} of $C$, will be denoted by $C^\circ$. It is set of
linear forms whose zero set meets the closure of $C$ in the origin only.
For a closed convex cone $C$ in $V$, we have $C^{**}=C$:  according to
Corollary 11.7.1 of \cite{rock} $C$ is the intersection of the half spaces which contain it
(if this collection of half spaces is empty, then this intersection  is $V$ by definition).

Let $A$ be a finite dimensional real affine space.
Given a convex subset $X\subset A$, then  the \emph{relative interior} of $X$ is the
interior of $X$ in its affine span; we denote it by $\mathring X$ (other authors use $\text{ri}(X)$). 
If this happens to be $X$, then we say that $X$ 
is \emph{relatively open}. 
A \emph{face} of $X$ is a nonempty subset $F$ of $X$ with
the property that every segment in $X$ which meets $F$ is
either contained in $F$ or meets $F$ in an end point.
A face is closed in $X$. Any intersection of faces of $X$ is a face of $X$
and the relative interiors of the faces decompose $X$.
If $f:X\to \RR$ is the restriction of an affine-linear function,
then the set of points of $X$ where $f$ assumes its
infimum is a face of $X$, but not every face of $X$ is
necessarily of that form. A face that can be so obtained 
is called \emph{exposed}. 
A point of $X$ that makes up a face resp.\ an exposed face
by itself is also called an \emph{extreme} resp.\ \emph{exposed point} 
of $X$. The \emph{star} of a face $F$ of $X$, $\Star_X(F)$, is the collection 
of the faces of $X$ that contain $F$. The union of the relative interiors
of theses faces will be denoted $|\Star_X(F)|$. In case $\Star_X(F)$ is finite,
$|\Star_X(F)|$ will be open in $X$, but this need not be so in general.
We say that $X$ is \emph{nondegenerate} if it
does not contain an affine line.

We denote by $\Tr(A)$ the vector space of translations of the affine space $A$
and extend the range of  that notation as follows. 
If $X$ and $Y$ are subsets $A$, then we denote by
$\Tr (X,Y)\subset \Tr (A)$ the set of translations that take $X$ to $Y$. 
In the special case when $X=Y$ is convex, then $\Tr (X,Y)$ is a convex cone (if 
$t\in \Tr (A)$  is such that $t+X\subset X$ and  $\lambda>0$, then choose
an integer $n\ge \lambda$; we have $nt+X\subset X$ and then also
$\lambda t+X\subset X$ by convexity), called the \emph{recession cone} of $X$, and  denoted
$\Tr (X)$ instead (Rockafellar \cite{rock} writes $0^+X$). The recession cone of a \emph{bounded} convex set
is clearly reduced to the origin. According to \cite{rock}, Thm. 8.4,
the converse holds for any convex subset that is closed (and hence also when it is relatively open). In that case $\Tr (X)$ has a simple  geometric  interpretation: if we compactify $A$ to a topological ball  by adding as  boundary the topological sphere  of directions in $\Tr (A)$, then the  closure of a closed convex $X\subset A$ intersects that boundary sphere in the subset defined by the directions in $\Tr (X)$. We can also express this as follows: if
$\tilde T(A)$ denotes the dual of the space of affine-linear functions on $A$, so that $A$ is naturally embedded in  $\tilde T(A)$ as an affine  hyperplane and $T(A)$ is embedded in $\tilde T(A)$ as  the linear hyperplane parallel to $A$, then the closure of the cone spanned by $X$ meets $T(A)$ in $T(X)$.

The following related notion is also useful.

\begin{definition} 
Let $X$ be a subset of a  finite dimensional real affine space $A$. Then the 
\emph{asymptotic space} of $X$ is the intersection of all the 
linear subspaces $W\subset \Tr (A)$ for which the image of $X$ in the affine quotient $A/W$ 
is bounded. 
We denote this subspace $\as(X)$.
\end{definition}

In other words, $\as (X)$ is the smallest linear subspace $W\subset \Tr (A)$ with the property
that the image of $X$ in $A/W$ is bounded.
Notice that we can also characterize $\as (A)$ as the common zero set of the 
linear parts  of the affine-linear $\RR$-valued functions 
on $A$ whose restriction to $X$ is bounded.

It is clear that $\as(X)$ contains $\Tr (X)$. But the latter need not span
the former. For instance, if $X$ is the solid parabola in $\RR^2$ defined by $y>x^2$,
then $\Tr (X)$ is the nonnegative $y$-axis, whereas $\as(X)=\RR^2$

We denote the convex hull of a subset $Z$ of an affine space by $[Z]$.  

\medskip
In the remainder of this section, $V$
is a finite dimensional vector space and $C$ an open
nondegenerate convex cone in $V$. Then $C^\circ$ is 
nondegenerate as well and we have $(C^\circ)^\circ=C$.

We first note that if $F$ is a face of $\bar C$, then 
the elements of $C^*$ that  vanish on $F$ make up a face  $F^\dagger$ of $C^*$. In fact, this is an exposed face of $C^*$: if $x\in\mathring F$, then if $\xi\in C^*$ (in other words, $\xi |C\ge 0$) is such that $\xi (x)=0$, then we must have $\xi |F=0$ and so $\xi \in F^\dagger$. This shows that $F^\dagger$ is exposed by $x$.

Let us  denote the annihilator of $F^\dagger$ in $V$ by $V^F$. So this is the common zero set of the $\phi\in C^*$ which vanish on $F$. Clearly, $F\subset V^F$, and the projection $\pi: V\to V/V^F$ is dual to the inclusion of the linear span of $F^\dagger$ in $V^*$. Under this duality, $\pi C$ can be identified with the open dual of $F^\dagger$. 
In particular, $\pi C$ is an open nondegenerate convex cone in $V/V^F$. It may be characterized as the biggest projection of $C$ that is a nondegenerate convex cone and has $F$ in the preimage of its vertex. 

In Section \ref{section:ratsets} we shall need the following lemma.

\begin{lemma}\label{lemma:recrec} 
Let $W\subset V$ be the asymptotic space of some subset of $\bar C$. Then for every nonempty compact  $B\subset C$, $W$ is the
asymptotic space of $\Tr (B, C)\cap W$; in other words, any $\phi\in W^*$ which is bounded on $\Tr (B, C)\cap W$ must be zero. 
\end{lemma}

The proof needs:

\begin{lemma}\label{lemma:ray} 
Let $F$ be a face of $\bar C$ and $q\in \mathring{F}$. Then for every compact $B\subset C+V^F$ there exists a
$\lambda > 0$ such that $B+\lambda q\in C$.
\end{lemma}

\begin{proof} 
We show that for every $p\in C+V^F$ there exists a $\lambda_p>0$ such that $p+\lambda_p q\in C$. This suffices, for then 
$p$ has a neighborhood $B_p$ in $V$ with the property that $B_p+\lambda_p q\subset C$ and the lemma follows from
the compactness of $B$.
So suppose that for a given $p\in C+V^F$  no such $\lambda $ exists. Then
$(p+\RR_{\ge 0} q)\cap C=\emptyset $. Now notice that  whenever $p'\in (p+\RR q)\cap \bar C$,
then $p'+\RR_{\ge 0} q\subset \bar C+\bar C=\bar C$. So  either $(p+\RR_{\ge 0} q)\cap \bar C$ is empty or is a ray  contained in $p+\RR q$.  In both cases there exists by  Thm.\ 11.6 of \cite{rock} a linear form $\xi$ on $V$
which is positive on $C$ and nonpositive on a ray in $p+\RR q$. This implies that $\xi (q)=0$ (and hence $\xi (p)\leq 0$).  Since $\xi |C>0$,
we have $\xi |V^F=0$, and hence $\xi |C+V^F>0$. But this contradicts the fact that $p\in C+V^F$ and $\xi (p)\leq 0$.
\end{proof}

\begin{proof}[Proof of \ref{lemma:recrec}]
We prove this with induction on $\dim V$. Choose a convex subset $P\subset \bar C$  such that $W=\as (P)$. We have $T(P)\subset\bar C$. Let $F$ be the (unique) face of $\bar C$ whose relative interior contains the relative interior  of $T(P)$. 
If $F=\{ 0\}$, then $T(P)=\{ 0\}$. This implies that $P$ is bounded and so $W=V$ and the lemma  clearly holds in that case. We therefore assume $F\not=\{ 0\}$.  Denote by $\pi :V\to V/V^F$ be the projection and observe that $\dim (V/V^F)<\dim V$. 
Suppose $\phi\in W^*$ is bounded on $\Tr (B, C)\cap W$. We must show that $\phi$ is zero.

First we prove that 
$\phi$  factors through a linear form $\phi':\pi W\to \RR$. Choose $x$ in the relative interior of $\Tr (P)$.
Then $x\in \bar C\cap W$ and so $\RR_{\ge 0}x\subset \Tr (B, C)\cap W$. Since $\phi$ is bounded on $\Tr (B, C)\cap W$, it follows that 
$\phi (x)=0$.
Following  Lemma \ref{lemma:ray} there exists for every $u\in V^F\cap W$ a $\lambda>0$ such that $\lambda x+u+B\subset C$, i.e.,  $\lambda x+u\in  \Tr (B, C)\cap W$. Since $\phi (u)=\phi (\lambda x+u)$ it follows that $\phi$ is bounded  on the linear subspace $V^F\cap W$. So $\phi |V^F\cap W=0$ and $\phi$ factors through $\pi W$.

Observe that $\pi W$ is the asymptotic subspace of  a subset of $\pi \bar C$: a linear form is bounded on $\pi P$ if and only if its pull-back to $V$ is bounded on $P$. We claim that 
\[
\pi (\Tr (B,C)\cap W)=\Tr (\pi B, \pi C)\cap \pi W.
\]
The  inclusion $\subset$ is clear, so let us prove $\supset$. Suppose $w\in W$ is such that $\pi (w)\in \Tr (\pi B, \pi C)$. This means
that $w+B\subset C+V^F$. According to Lemma \ref{lemma:ray} there exists a $u\in W\cap V^F$ such that $u+w+B\subset C$ and so $u+w\in \Tr (B,C)\cap W$. Since $\pi (w)=\pi (u+w)$, the claim follows.

Our induction hypothesis can now be applied to $\pi C$, $\pi W$ and $\pi B$ and this enables us to  conclude that $\phi'=0$. 
\end{proof}

\begin{definition}
A \emph{kernel} for $C$ is a nonempty
convex subset $K$ of $\bar C$ with $0\notin \bar K$
and $K+C\subset K$. (N.B. Our definition differs
slightly from the one of Ash \cite{amrt}, in that he does
not insist that $K$ be convex.) Two kernels $K_1$
and $K_2$ for $C$ are said to be \emph{comparable} if
$\lambda K_1\subset K_2\subset \lambda ^{-1}K_1$ for
some $\lambda >0$. 
\end{definition}

It is clear that comparability is an equivalence relation.

For any subset $K$ of $V$ we let $K^{\v}$ denote
the set of $\xi \in V^*$ with $\xi |K\geq 1$. If
$K$ is a kernel for $C$, then it is easy to see that
$K^{\v}$ is a closed kernel for $C^{\circ }$.
Furthermore, if $K_1$ and $K_2$ are comparable
kernels, then so are $K^{\v}_1$ and $K^{\v}_2$.

\begin{lemma} 
If $K$ is a kernel for $C$,
then $K^{\v\v}=\bar K$.
\end{lemma}
\begin{proof} Although this is essentially \cite{rock} 
II.5.2 Prop.1, we give a proof for completeness. Let us prove
the nontrivial inclusion $K^{\v\v}\subset \bar K$.  If
$x\notin \bar K$, then there exists by the
separating hyperplane theorem \cite{rock}, Thm. 11.5, a $\xi
\in V^*$ with $\xi (x) < \inf_K \xi $. It is clear that then
$\inf \xi | {\bar C} \geq 0$ (and hence $\inf \xi |K  \geq 0$). If
$\inf \xi|K >0$, then normalize $\xi $ such that
$\inf \xi | K=1$. 
Then $\xi \in K^\v$, and since $\xi (x)<1$, we have
$x\notin K^{\v\v}$. Suppose now $\inf \xi|K =0$
(so that $\xi (x)<0$). The preceding
argument applied to $x=0$ yields a $\xi _0\in V^*$
with $\inf \xi_0|K=1$. Choose $t>-\xi_0(x)/\xi
(x)$. Then $(\xi _0+t\xi )(x)<0$ and $\xi _0+t\xi
\in K^\v$. So in this case $x\notin K^{\v\v}$ also.
\end{proof}

Closed kernels have certain  technical advantages 
over arbitrary ones and for that reason
the following lemma is quite useful. 

\begin{lemma}\label{lemma:facial} 
Let $\Lambda$ be a discrete subset of $\bar C$.
Then the convex set  $[\Lambda ]+\bar C$ is closed in $V$
and every extreme point of $[\Lambda ]+\bar C$ is
exposed and belongs to $\Lambda $. Moreover, every
face of $[\Lambda ]+\bar C$ is of the form $[M]+F$,
where $M\subset \Lambda$ and $F$ is a face of $\bar C$.
\end{lemma} 

\begin{proof} Let $K$ denote the closure of $[\Lambda ]+\bar
C$. We first show that every exposed
point $p$ of $K$ is in fact in $\Lambda$. By 
definition there exists a $\xi\in V^*$ 
with $\xi (p)<\xi |K-\{ p\}$. Since $p+\bar C\subset K$, it follows
that $\xi$ is positive on $\bar C-\{ 0\}$. This means that
$\xi |\bar C$ is proper. Since $\Lambda$ is closed in $\bar C$, it follows
that $\xi |\Lambda$ has a minimum. This is then also the minimum 
of $\xi |K$, and so we must have  $p\in\Lambda$. 

Following Straszewicz's theorem \cite{rock}, Thm.\ 18.6, the exposed
points of $K$ are dense in the set of extreme points
of $K$. But $\Lambda $ is discrete and so
every extreme point of $K$ must be an exposed point of $K$.
It now follows from \cite{rock}, Thm.\ 18.6, that $K\subset
[\Lambda ]+\bar C$. The last assertion is a
consequence of \cite{rock}, Thm. 18.5. 
\end{proof}

\section{Convex Cones and lattices}\label{section:ratsets}

The main result of this section is Theorem \ref{thm:bdlypol} below, which
may not strike the reader as surprising. Nevertheless we shall see that
it has interesting consequences, such as the 
\emph{Siegel property} \ref{prop:siegel}. 

Before we can state the result alluded to above, we recall  resp.\  introduce 
some terminology pertaining to polyhedra.

\begin{definition}
A \emph{polyhedron} in a real finite dimensional 
affine space is a subset of that can be defined by finitely many affine-linear
nonstrict inequalities  (so of the form $f\le 0$). If the affine space has a $\QQ$-structure
and these affine-linear forms are definable
over $\QQ$, then we call this a \emph{rational} polyhedron. A subset of the 
affine space is said to be \emph{(rationally) locally polyhedral} if its
intersection with every bounded (rational) polyhedron
is a (rational) polyhedron. 
\end{definition}

We note here that every \emph{bounded} polyhedron is the convex hull of a finite set (in some affine 
space), and is rational if we can take that subset to consist of rational points.

\begin{theorem}\label{thm:bdlypol}
Let $V$ be a real finite dimensional vector space, $C\subset V$ an open
nondegenerate convex cone and $L\subset V$
a lattice. Then $[C\cap L]$ is locally rationally polyhedral in 
$V$ (relative to the $\QQ$-structure on $V$ defined by $L$). In particular, 
$[C\cap L]$ is closed in $V$.
\end{theorem}

It will be convenient to prove a few preparatory results first. 

\begin{lemma}\label{lemma:halflinenbhd}
Let $p_0\in L$, $B$ be a neighborhood of  $p_0$  in $V$ and $R$ a relatively open half line in $V$ that is not contained in  a proper  linear subspace of $V$ defined over $\QQ$.  Then $[(B+R)\cap L]$ is a \nbhd\ of $p_0+R$ in $V$.
\end{lemma}
\begin{proof}
Without loss of generality we may and will assume that $p_0=0$.
It is a well-known fact that the image of a line in $V$ not 
contained  in  a proper linear $\QQ$-subspace of $V$ has dense
image  in $V/L$. The same is true for a ray in such a line, 
such as $R$.

Let $x\in R$. We show that $[(B+R)\cap L]$ is a \nbhd \ of $x$ in $V$, more precisely,
we will find  $p_1,\dots , p_N\in (B+R)\cap L$ such that $x$ is in the interior of the 
convex hull of $p_0=0, p_1,\dots , p_N$.
For this we choose nonempty open subsets $U_1,\dots ,U_N$ of $B+R$ 
such that for every $(u_1,\dots ,u_N)\in
U_1\times \cdots \times U_N$, $x$ is in the
 interior of the convex hull of $u_1,\dots ,u_N$.
Since $\rho (U_i+R)=V/L$, there exist $u_i\in U_i$ and $t_i>0$ with 
$p_i:=u_i+t_ix \in (U_i+R)\cap L$, $i=1,\dots ,N$. 
We can write $x$ as a strictly convex linear combination of $u_1,\dots ,u_N$:
$x=\sum_i\lambda_iu_i$ with all $\lambda_i>0$, $\sum_i\lambda_i=1$.
Then $\sum_i \lambda_ip_i=(1+\sum_i \lambda_it_i)x$, and so we have
$x=\sum_{j=0}^N \mu_j p_j$ with $\mu_j=\lambda_i (1+\sum_i \lambda_it_i)^{-1}$ when $j>0$ and
$\mu_0=(\sum_i \lambda_it_i)(1+\sum_i \lambda_it_i)^{-1}$. 
Hence $p_1,\dots , p_N$ are as desired.
\end{proof}

In what follows we denote by $\rho : V\to V/L$ the obvious map.

\begin{lemma}\label{lemma:rho}
Let $W$  be a subspace of $V$ defined over $\QQ$ and $Y\subset W$  a convex subset  such that $W$ is the smallest 
subspace of $V$ defined over $\QQ$ that contains $\as (Y)$. 
Then the closure of $\rho (Y)$ equals the subtorus $\rho (W)=W/L\cap W$ of $V/L$.
\end{lemma}
\begin{proof}
If $Y$ is bounded, then $\as (Y)=0$ and there is nothing to show. We therefore assume $Y$ unbounded. Without loss of generality we may assume that $Y$ is relatively open  so that $\Tr (Y)\not=0$. Choose $v_1$ in the relative interior of $\Tr (Y)$. Then for every $y\in Y$, $\rho(y+\RR_{\ge 0} v)$  is dense in $\rho(y+\RR v)$ and so $\rho (Y)$ is dense in the image of $Y_1:=Y+\RR v_1$. If $\as(Y)\not=\RR v_1$, we proceed in this manner and we eventually find that $\rho (Y)$ is dense in $\rho (Y+\as (Y))$.  Now $\rho(\as (Y))$ is a connected subgroup of $V/L$ and it is well-known that the closure of such a group must be a subtorus. In the present case, this must be  $W/L\cap W$.
Since $Y\subset W$, it then also  follows that $\rho (Y+\as (Y))$ is dense in $\rho (W)$.
\end{proof}

\begin{proof}[Proof of theorem \ref{thm:bdlypol}]
We prove the theorem with induction on $\dim V$. Since
there is nothing to show when $V=\{ 0\}$, we assume that $\dim V>0$.
The convex set $[C\cap L ]+\bar C$ is closed in $V$ by \ref{lemma:facial} and 
contained in the open cone $C$ (since $C+\bar C\subset C$).
Let $P$ be a face of $[C\cap L ]+\bar C$.

\medskip
\emph{Step 1. $P$ is also a face of $[C\cap L ]$. In particular, the affine span of
$P$ is defined over $\QQ$.}

\begin{proof} According to Lemma \ref{lemma:facial} (applied to $\Lambda:=L\cap C$), we have 
$P=[P\cap L]+F$, with $F$ a face of $\bar C$ and $P\cap L\subset C$. So it is enough to prove that
for every $p\in P\cap L$ and every ray $R$ in $F$, $p+R\subset [C\cap L]$.
Denote by $V_R$  the smallest subspace of $V$ defined over 
$\QQ$ which contains $R$ and let $B$ be a convex \nbhd\ of $0$ in
$V_R$ such that $p+B\subset C$. According to Lemma \ref{lemma:halflinenbhd},  
$[(p+B+R)\cap L]$ contains a neighborhood of $p+R$ in $V_R$.
Since $p+B+R\subset C$ it follows that $p+R\subset [C\cap L]$.
\end{proof} 

Note that if we apply Step 1 to $P=[C\cap L ]+\bar C$, we  find that 
$[C\cap L]=[C\cap L ]+\bar C$ is closed in $V$.

We write  $W$ for the smallest subspace of $V$ defined over $\QQ$
that contains $\as (P)$ and $\pi :V\to V/W$ for  the projection. So $\pi P$ is bounded
and $\pi L$ is a lattice in $V/W$.

\medskip
\emph{Step 2. If $Q$ is face of $[C\cap L]$ which contains $P$,
then $Q=\pi ^{-1}\pi Q\cap [C\cap L ]$ and hence $\pi Q$ is a face of $\pi [C\cap L]$.}

\begin{proof} 
Clearly, the translation space of the affine span $\aff (Q)$ of $Q$  contains $\as (P)$. 
Since the former is defined over $\QQ$, it contains  $W$ as well. This implies that $\pi ^{-1}\pi Q\cap [C\cap L ]$ equals  $\aff (Q) \cap  [C\cap L ]$. Since  $Q$ is closed, the latter is just $Q$. This implies  that $\pi Q$ is a face of $\pi [C\cap L]$  indeed: any segment in $\pi [C\cap L]$ is of the form  $\pi \sigma$ for some segment   $\sigma$ in $[C\cap L]$, and if $\pi\sigma$ meets $\pi Q$, then $\sigma$ meets 
$\pi ^{-1}\pi Q\cap [C\cap L ]=Q$ and we have $\pi\sigma\cap \pi Q=\pi (\sigma\cap Q)$. So
the latter is either $\pi\sigma$ or an end point of $\pi\sigma$. 
\end{proof} 

\emph{Step 3. $[\pi C\cap \pi L]=\pi [C\cap L]$.}
\begin{proof} 
We prove the nontrivial inclusion $\subset$, that is,
we show that if $p\in C$ is such that $(p+W)\cap L$ is nonempty,
then $(p+W)\cap C\cap L$ is nonempty.

Choose a compact convex neighborhood $B$ of $p$ in $(p+W)\cap C$.  By Lemma \ref{lemma:recrec} , $T(B,C)\cap \as (P)$ and $P$ have the same asymptotic space.
Since $T(B,C)\cap \as (P)$ is convex, Lemma \ref{lemma:rho} applies here and tells us that the closure of $\rho (T(B,C)\cap \as (P))$ is $\rho W$. 

Now $B$ contains a nonempty open subset of $p+W$ and hence of $p+W+L=W+L$. 
The same is the true of $-B$ and so $\rho (-B)$ meets $\rho (\Tr (B,C)\cap \as (P))$. In other words, there exist $b\in B$ and 
$v'\in\Tr (B,C)\cap \as (P)$ such that $b+v'\in L$. It is clear that we also have $b+v'\in C\cap (p+W)$. 
\end{proof} 

\emph{Step 4. $\pi C$ is nondegenerate.}
\begin{proof} 
 Suppose not: then $\Tr (\pi C)$ 
contains a line $l\subset V/W$.
Choose $p\in P\cap L$. It follows from Lemma \ref{lemma:halflinenbhd} 
(applied to the two rays in $\pi(p)+l$ emanating from $\pi(p)$) that $\pi (p)+l$ is in the convex hull of $\pi C\cap \pi L$.
By step 3, this convex hull is just $\pi [C\cap L]$. 
Let $y\in l-\{ 0\}$. Since $\pi P$ is bounded,
there exists a $\mu >0$ such that $\pi (p)\pm\mu y\notin \pi P$. 
Let $p_{\pm }\in [C\cap L]$ be such that $\pi (p_{\pm })=\pi (p)\pm \mu
y$. Then $\frac{1}{2}p_-+\frac{1}{2}p_+\in \pi ^{-1}\pi
(p)\cap [C\cap L ]=P$ by step 2, although $p_{\pm
}\notin P$. This contradicts the fact that $P$ is a
face of $[C\cap L]$.     
\end{proof} 

\emph{Step 5. If $P$ is unbounded, then $\Star_{[C\cap L]}(P)$
has only finitely many members and $|\Star(P)_{[C\cap L]}|$ is
a neighborhood of $\mathring P$ in $[C\cap L]$.}

\begin{proof}   Since $P$ is unbounded, $W\not= \{ 0\}$, and so
$\dim V/W <\dim V$. Step 4 enables us to apply our
induction hypothesis to $\pi C$ and $\pi L$. Since
$[\pi C\cap \pi L]=\pi [C\cap L]$ (step 3), we find
that $\pi [C\cap L ]$ is a locally \rp \ subset
of $\, V/W$. So $\Star_{\pi [C\cap L]} (\pi P)$
is finite.  This implies that 
$|\Star_{\pi [C\cap L]} (\pi P)|$ is a neighborhood
of the relative interior of $\pi (P)$ in $\pi ([C\cap L ])$ and that each member of 
$\Star_{\pi [C\cap L]} (\pi P)$ is obtained as in Step 2, i.e., is the image 
of face of $[C\cap L ]$ which contains $P$ and which is also a face of  $[C\cap L ]+\bar C$ .
This implies the corresponding property for $P$ with respect to
$[C\cap L ]$. 
\end{proof} 

\emph{Step 6 (Conclusion).} 
We show that any  bounded  polyhedron $\Pi$ in $V$ meets only 
finitely faces of $[C\cap L  ]$.
Suppose that on the contrary, there exists a sequence
$P_1,P_2,\dots $ of pairwise distinct faces of $[C\cap L ]$
with $P_i\cap \Pi\not= \emptyset$. Clearly,  
$\cup^{\infty }_{i=1}P_i$ cannot be bounded: otherwise 
$(\cup^{\infty }_{i=1}P_i)\cap L$ would be finite, and as
each $P_i$ is the convex hull of its intersection
with $L$, only a finite number of $P_i$'s could be
distinct. This property is of course also true for the union over any subsequence
of $(P_i)_i$. So, perhaps after  passing to a subsequence, we can find
sequences $\{ p_i\in P_i\cap \Pi \} ^{\infty }_{i=1}$
converging to some $p_\infty\in \Pi$ and $\{ q_i \in
\mathring P_i\} ^{\infty }_{i=1}$ such that the intervals $[p_i,q_i]$
converge to a closed half line  emanating from $p_\infty$.
Let $P_\infty$ denote the face of 
$[C\cap L ]$  whose relative interior contains the relative interior
of this half-line.
As $P_{\infty}$ is unbounded, $|\Star_{[C\cap L]}(P_\infty)|$ is by
step 5 a \nbhd \ of $\mathring P_{\infty }$, and so
for $i$ sufficiently large, $q_i\in |\Star(P_\infty)|$ and hence
$P_i\supset P_{\infty }$. According to
step 5 only finitely many faces of $[C\cap L ]$ have
that property, and so we get a contradiction.
\end{proof}

\begin{remark}\label{rem:1}
If it were true that a nondegenerate convex cone which is  the linear projection of a closed nondegenerate cone is also closed, then the above proof  could be simplified. That is however not so, as the following example due to Burt Totaro shows. The function $x\in (-1,1)\mapsto (1-x^2)^{-1}$  is convex and so  the open subset $D\subset \RR^2$ defined by $y>0$ and  $y (1-x^2)> 1$ is convex as well. Note that if we put $D$ in real projective space, then $\partial D$ is smooth except for its unique point at infinity: there it has two distinct tangent lines (given by $x=\pm 1$) which do not meet $\bar D$ outside that point. We now take for $C\subset \RR^3$ the cone over $\bar D$ and project along the axis defined by its point at infinity. 
So $C$  is given by the inequalities $|x|\le z$, $y\ge 0$, $y(z^2-x^2) \ge  z^3$, and we project $C$  under $(x,y,z) \mapsto (x,z)$. Then the image of $C$ is the union of the origin $(0,0)$ and the \emph{open} cone  defined by $|x| < z$. This is a nondegenerate convex cone, but it is not closed.
\end{remark}

\section{The Siegel Property}\label{section:ratcones}

In this section $V$ is a real finite dimensional vector
space with a $\QQ$-structure $V(\QQ)$. A \emph{lattice} in $V$
is always understood to be compatible with this
$\QQ$-structure, in other words, must be a subgroup of $V(\QQ)$ of rank
equal to $\dim V$.

We also fix an open convex nondegenerate cone $C\subset V$.
The convex hull of $\bar C\cap V(\QQ)$ clearly
contains $C$; we denote it by $C_+$. Similarly we
have $C_+^{\circ }\supset C^{\circ }$ in $V^*$. 

We say that a
subset $K\subset V$ is \emph{\lrp\  in $C_+$} if for every
\rp \ cone $\Pi$ in $C_+$, $\Pi \cap K$ is a
rational polyhedron (note that we do not require that $K\subset C_+$). 

\begin{proposition}  
For every lattice $L$, $[C\cap L]$ is a kernel for $C$ 
and all such kernels belong to the same comparability class.
\end{proposition}
\begin{proof}
It is easy to see that $[C\cap L]$ is a kernel. 
If $L'$ is another lattice, then there exists a positive integer
$k$ such that $kL\subset L'\subset \frac{1}{k}L$. So
$k[C\cap L]\subset [C\cap L']\subset \frac{1}{k}[C\cap L]$.
\end{proof}

\begin{definition}\label{def:core}
A kernel for $C$ is called a \emph{core} if it is comparable with
the convex hull of the intersection of $C$ with some lattice in $V$.
It is called a \emph{cocore} if its dual is a core for $C^\circ$. 
Notice that given a lattice $L$ in $V$, then 
$[C^\circ\cap L^*]^\v$ is a cocore which contains the core $[C\cap L]$.
\end{definition}     

\begin{definition}\label{def:lrp}  
A collection $\Sigma $ of
convex cones in $C_+$ is said to be a \emph{\lrp \
decomposition} of $C_+$ if the following
conditions are fulfilled:
\begin{enumerate}
\item[(i)] the relative interiors of the members of $\Sigma$ are pairwise disjoint and their union is $C_+$,
\item[(ii)] $\Sigma $ is closed under intersections and taking faces,
\item[(iii)] if $\Pi $ is a \rp \ cone in $C_+$, then
$\Sigma |\Pi:=\{ \sigma \cap \Pi\}_{\sigma \in
\Sigma}$ is a finite collection of \rp \ cones.
\end{enumerate}
If moreover every $\sigma \in \Sigma $ is a \rp \
cone, then we omit ``locally'', and call
$\Sigma $ a \emph{\rp \ decomposition} of $C_+$.
\end{definition} 

\begin{remark}
The collection of faces of $C_+$ is obviously a \lrp \
decomposition of $C_+$. It is in fact the coarsest as it is refined by any other
\lrp \ decomposition of $C_+$. 
\end{remark}

\begin{proposition}\label{prop:rpdec}
Let $L\subset V(\QQ)$ be a lattice.
For any face $P$ of $[C^{\circ
}\cap L^*]$, let  $\sigma (P)$ be the set of $x\in
V$ such that $\xi \in [C^{\circ }\cap L^*]
\mapsto \xi (x)$ assumes its infimum on
all of $P$. Then $\sigma (P)$ is a \rp \ cone of dimension equal to
$\codim P$, $P\mapsto \sigma (P)$ is an
injection which reverses inclusions, and 
$\Sigma (C,L):=\{
\sigma (P)\}_P $ is a \rp \ decomposition of $C_+$.
\end{proposition}

Before we begin the proof we show:

\begin{proposition}\label{prop:cocore} 
If $L$ is a lattice in $V(\QQ)$, then $[C^{\circ }\cap
L^*]^\v$ is a \lrp \ cocore for $C$.
\end{proposition} 

This will be a consequence of the following result 
(which we shall later need it in this general form
for the proof of Proposition \ref{prop:lrpkernel}):

\begin{lemma}\label{lemma:finitepol} 
Let $A$ be a real finite dimensional affine space, $P$ a
polyhedron in $A$, and $\Phi$ a collection of affine-linear
functions on $A$ that are $\ge 0$ on $P$ and such that for every $p\in P$ and every $t\in \Tr (P)$, 
the sets  $\{ \phi (p)\}_{\phi \in \Phi}$ and $\{ d\phi (t)\}_{\phi \in \Phi}$ are discrete. For any finite subset $S\subset\Phi$, denote by
$P_S$ the set of $p\in P$ such that all $\phi\in S$
assume in $p$ the same value and no member of $\Phi$
takes in $p$ a smaller value.
Then $\{ P_S\}_{S\subset \Phi \text{ finite}} $ is a finite polyhedral
decomposition of $P$ and $P_{\Phi\ge 1}:=\{ p\in P\, |\, \phi(p)\geq 1 \text{ for
all $\Phi$}\} $ is a polyhedron.

If $A$ is endowed with a $\QQ$-structure and relative to this structure, $P$ is a rational polyhedron and the members of $\Phi$ are defined over $\QQ$, then the $P_S$ and $P_{\Phi\ge 1}$ are rational polyhedra.
\end{lemma} 

\begin{proof}
Let $p_1,\dots ,p_k$ enumerate the (finite) set of 
extreme points of $P$, and let $t_{k+1},\dots ,t_l\in
\Tr (P)$ generate $\Tr (P)$ as a cone. Consider
the subset  of $\RR^l$ defined by 
\[
\Xi:=\{ (\phi (p_1),\dots ,\phi (p_k),d\phi (t_{k+1}), \dots d\phi (t_l))\,
|\, \phi \in \Phi\} .
\]
By assumption the projection of $\Xi $ on every coordinate is
discrete and contained in $\RR_{\geq 0}$. An inductive argument
shows that there exists a finite subset $\Xi_0$ of
$\Xi $ such that $\Xi \subseteq \Xi _0+\RR_{\geq 0}^l$. So if
$\Phi_0$ is a finite subset of $\Phi$ which maps onto 
$\Xi _0$, then, for every $\phi \in \Phi$ there exists a
$\phi _0\in \Phi_0$ such that for all $i,j$, $\phi (p_i)\geq \phi _0(p_i)$ and
$d\phi (t_j)\geq d\phi _0(t_j)$. In other words $\phi \geq \phi _0$ on
$P$. Hence every nonempty $P_S$ is obtained by taking 
$S\subset \Phi_0$, and all such cover $P$. The set of $p\in P$
with  $\phi(p)\geq 1$ for all $\phi \in \Phi$ is already defined
by restricting the index set to $\Phi_0$ and is therefore a polyhedron.

The $\QQ$-version is then straightforward.
\end{proof}

\begin{proof}[Proof of Proposition \ref{prop:cocore}]
Let $\Pi \subset C_+$ be a rationally  polyhedral cone in $C_+$.
We must show that $\Pi \cap [C^{\circ }\cap L^*]^\v$, that is, the locus
of $p\in \Pi$ with $\xi (p)\ge 1$ for all $\xi\in C^\circ\cap L^*$, is rationally polyhedral.
It is enough to show that the
($\QQ$-version of) Lemma \ref{lemma:finitepol} applies here with $P=\Pi$ and $\Phi=C^\circ\cap L^*$. If $p\in C\cap L$,
then clearly the set of $\{ \xi (q)\}_{\xi\in C^\circ\cap L}$ is a
set of positive integers. So if $p$ is a convex linear combination of such
$q$, then $\{ \xi (p)\}_{\xi\in C^\circ\cap L}$ is in a
semigroup of $\RR_{\ge 0}$ which is finitely generated as such. Hence it is discrete
as a subset and bounded from below.
\end{proof}

\begin{proof}[Proof of Proposition \ref{prop:rpdec}]
We only prove the last part of the  statement, for
everything else follows in a straightforward manner
from \ref{thm:bdlypol}.
For $p\in \bar C\cap L$, the set 
$\{ \xi (p)\}_{ \xi
\in C^\circ\cap L^*}$ obviously consists of positive integers.
So for $x\in [\bar C\cap L]=C_+$, 
$\{ \xi (x)\}_{\xi \in C^\circ\cap L^*}$ is still a
discrete subset of $\RR_{\geq 0}$. This implies that the function 
$\xi \in [ C^{\circ }\cap L^*]\mapsto \xi (x)$
has a minimum, so that $x \in \sigma (P)$ for some
face $P$ of $[ C^{\circ }\cap L^*]$. This proves 
property (i) of \ref{def:lrp}. Property (ii) is easy.
As for (iii), let $\Pi $ be a \rp \ cone in $C_+$.
Then $\Pi \cap [C^{\circ }\cap L^*]^\v$ is a rational 
polyhedron by  \ref{prop:cocore}. Since 
$\sigma (P)\cap \Pi$ is the cone over a face of $\Pi
\cap [C^0\cap L^*]^\v$ or reduced to the origin, the collection 
$\{ \sigma (P)\cap \Pi\}_P$ is finite and consists of \rp \ cones.
\end{proof}

Here is an interesting application.

\begin{theorem}[(Siegel property)]\label{prop:siegel}
Let $\Gamma $ be a subgroup of $\GL (V)$ which
leaves $C$ and a lattice in $V(\QQ)$ invariant. Then 
$\Gamma$ has the \emph{Siegel property} in $C_+$: if $\Pi
_1$ and $\Pi _2$ are polyhedral cones  in $C_+$, then the collection 
$\{ \gamma \Pi _1 \cap \Pi _2\}_{\gamma \in \Gamma}$ is finite.
Moreover, if $F_i$ denotes the face of $C_+$ whose relative interior contains 
the relative interior of $\Pi_i$,
then the set of $\gamma \in
\Gamma$ with $\gamma \mathring \Pi _1 \cap \mathring \Pi _2\not=
\emptyset $ is a finite union of right $Z_{\Gamma
}(F_1)$-cosets (hence also a finite union of left
$Z_{\Gamma }(F_2)$-cosets). In particular, $\Gamma $
acts properly discontinuously on $C$.
\end{theorem}

\begin{proof} The first assertion follows from the second if
we apply it to the relative interiors of the (finitely
many) faces of $\Pi_1$ and $\Pi_2$. In order to
prove the second assertion, we first observe that the
set of $\gamma \in \Gamma $ with $\gamma \mathring \Pi
_1\cap \mathring \Pi _2 \not= \emptyset $ is indeed a
union of right $Z_{\Gamma }(F_1)$-cosets (and also of
left $Z_{\Gamma }(F_2)$-cosets).
Let $L\subset V(\QQ)$ be a $\G$-invariant lattice.
Since $\mathring \Pi _i$ is covered by finitely many
members of $\Sigma (C,L)$ whose relative interiors are
contained in $\mathring F_i$, we
may assume that $\Pi _i\in \Sigma (C,L)$, $i=1,2$. If
now $\gamma \mathring \Pi _1\cap \mathring \Pi _2\not=
\emptyset$, then we must have $\gamma \Pi _1=\Pi _2$.
So we only need to show that $G_i:=N_{\Gamma }(\Pi_i)/Z_{\Gamma }(F_i)$ 
is finite. Since $\Pi_i$ meets the relative 
interior of $F_i$, the collection $\Star_{\Sigma (C,L)}(\Pi)|F_i$ of members of  
$\Sigma(C,L)|F_i$ containing $\Pi_i$ must be finite. This is clearly acted on
by $G_i$. It then suffices to see that some member $\Pi'_i$ of  $\Star_{\Sigma (C,L)}(\Pi)|F_i$
has finite $G_i$-stabilizer. This is clear
if we take $\Pi'_i$  to be maximal, i.e.,  with  the property that  
its relative interior is open in $F_i$, for then this stabilizer will act faithfully on 
$\Pi'_i$ and hence (since $\Pi'_i$ is a rationally polyhedral cone)  must be finite. 
\end{proof}

\section{Pairs of Polyhedral Type}

In this section $V$ continues to denote a real finite 
dimensional vector space equipped with a rational
structure $V(\QQ)\subset V$ and $C$ is an open
nondegenerate convex cone in $V$.

\begin{propdef}\label{prop:poltype} 
Let $\Gamma $ be a subgroup of $\GL (V)$ which
stabilizes $C$ and some lattice in $V(\QQ)$. Then
the following conditions are equivalent:
\begin{enumerate}
\item[$(i)$] There exists a polyhedral cone $\Pi $ in 
$C_+$ with $\Gamma \cdot \Pi = C_+$.
\item[$(ii)$] There exists a polyhedral cone $\Pi $ in 
$C_+$ with $\Gamma \cdot \Pi \supset C$.
\item[$(iii)$] For \emph{every} $\Gamma $-invariant lattice
$L\subset V(\QQ)$, $\Gamma $ has finitely many
orbits in the set of extreme points of $[C\cap L]$.  
\item[$(iv)$] For \emph{some} $\Gamma $-invariant lattice
$L\subset V(\QQ)$, $\Gamma $ has finitely many
orbits in the set of extreme points of $[C\cap L]$.
\item[$(i)^*$-$(iv)^*$] The corresponding property
for the contragradient action of $\G$ on $C^{\circ }$.
\end{enumerate}
Moreover, in case $(ii)$ we necessarily have $\Gamma
\cdot \Pi =C_+$.
If one of these equivalent conditions is fulfilled, we say that $(V(\QQ), C,\Gamma)$ 
is a \emph{polyhedral triple} or simply, that $(C_+, \G)$ is  of \emph{polyhedral type}.
\end{propdef} 

\begin{proof}
The implications $(i)\Rightarrow (ii)$ and $(iii)\Rightarrow (iv)$ 
are obvious.

We prove $(ii)\Rightarrow (iii)$. Let $\Pi $ be as in $(ii)$.
Without loss of generality we may assume that $\Pi $
is \rp  .  Let $S$ denote the set of extreme points of
$[C\cap L]$. Then we must show that $S \cap
\Pi$ is finite. Every extreme point of $[C\cap L]$ is
in $C\cap L$ and hence $S\cap \Pi \subset C\cap L\cap
\Pi $. 
Let $v_1,\dots , v_r$ denote the set of primitive integral 
generators of the extremal rays of $\Pi$. Any $e\in S\cap \Pi $ 
has the property that $e-v_i\notin C\cap \Pi$ for all $i$. This
implies that if we write $e=\sum_{i=1}^r \lambda_i v_i$ with
$\lambda_i\ge 0$, then $\lambda_i\le 1$ for all $i$. So $S\cap \Pi $ 
is contained in a compact set (a continuous image of $[0,1]^r$) and hence finite.

Proof of $(iv)\Rightarrow (i)^*$.
In \ref{prop:rpdec} we have defined  a
\rp \ decomposition $\Sigma :=\Sigma (C^\circ ,L^*)$ of
$C^{\circ }_+$. This decomposition is $\Gamma$-invariant, 
and the correspondence $P\mapsto \sigma
(P)$ between faces of $[C\cap L]$ and members of
$\Sigma $ is equivariant. Now extreme points of
$[C\cap L]$ correspond to maximal members of $\Sigma$. 
So if $S$ is a system of $\Gamma
$-representatives in the collection of extreme
points of $[C\cap L]$, then $\sum_{e\in S}\sigma (\{ e\})$ is a \rp \ cone in
$C^{\circ }_+$ whose $\Gamma$-orbit equals $C^{\circ }_+$.

These implications, together with their dual forms,
prove the equivalence of $(i)$ through $(iv)^*$. As
for the last assertion, we choose a \rp \ cone 
$\Pi _1\subset C_+$ such that $\Gamma \cdot \Pi
_1=C_+$ (which exists in view of $(ii)\Rightarrow
(i)$) and prove that $\Gamma \cdot \Pi \supset \Pi
_1$. By the Siegel property \ref{prop:siegel}, the collection $\{
\gamma (\Pi )\cap \Pi _1|\gamma \in \Gamma \} $ has
only finitely many distinct members, so $(\Gamma \cdot
\Pi )\cap \Pi _1$ is closed. Since $\Gamma \cdot \Pi
\supset C$ and the latter contains the interior of $\Pi _1$, it follows that $\Gamma
\cdot \Pi \supset \Pi_1$.
\end{proof}

An important class of examples is singled out by the
proposition below, which is essentially due to A. Ash \cite{amrt}.

\begin{proposition}\label{prop:groupexample} 
Let $\Gcal$ be a reductive $\QQ$-algebraic subgroup 
of the general linear group of $V$.
Assume that $C$ is an orbit of the identity component $G$ of  $\Gcal(\RR)$. 
Then $\G:=G\cap \GL (L)$ is an arithmetic group
in $\Gcal$ and  $(C_+, \G)$ is of polyhedral type.
\end{proposition}

\begin{proof}[Sketch of proof]
The $G$-stabilizer of a point of $C$ is maximal compact subgroup of $G$ so that
$C$ is in fact the symmetric space of $G$.
The reduction theory for arithmetic groups shows that a fundamental
domain for the arithmetic  group $\G$ acting in $C$ is contained in a finite union of 
so-called \emph{Siegel sets} in $C$ (see for instance \cite{amrt} for the definition).
Hence it suffices to show that a Siegel set is contained
in a rational polyhedral cone in $C$. This is however the easy part of the proof of  
II-Theorem 4.1 of \cite{amrt}.
\end{proof}

\begin{example}
Another interesting class of examples not 
contained in the one above arises in the theory of Coxeter groups. 
Let $(n_{ij})$ be a \emph{nonsingular}, integral
$l\times l$ generalized Cartan matrix \cite{looij:genroot} without
components of finite type. Let $W\subset \GL_l(\ZZ)$
be the (Weyl) group generated by the reflections
$s_i :x\mapsto x-\sum_j n_{ij}x_j$ and let $I$
denote the $W$-orbit of the \emph{fundamental chamber} 
$\Pi\subset \RR^l$ defined by $x_i\geq 0, i=1,\dots ,l\}$.
It is known that $I$ is nondegenerate convex
\cite{looij:genroot}, and so $(I,W)$ is of polyhedral type. 
The dual construction (for the contragradient action of $W$ on
$(\RR^n)^*$) yields a nondegenerate convex cone
$\check I\subset (\RR^n)^*$ which with $W$ also
forms a pair of polyhedral
type. But its closure is  in general \emph{not} the dual of $I$. 

Somewhat more general situations, which have been
investigated by Vinberg \cite{vin}, also give examples of
polyhedral triples.
\end{example}

\begin{example}
Algebraic geometry can provide interesting  and highly nontrivial examples of  polyhedral triples. 
If $X$ is a complex compact manifold, then 
take for $V$ the N\'eron-Severi group of $X$ tensored with $\RR$, for 
$C$ the cone in $V$  spanned by the ample classes (we assume this set to be nonempty) 
and for $\G$ the image $\aut (X)$ in $\GL(V)$. It is known that  $(C,V,\G)$ is a polyhedral triple 
for many  surfaces, among them K3 surfaces (Sterk \cite{sterk}) and Enriques surfaces
(Namikawa \cite{namikawa}). David Morrison's  cone conjecture  \cite{morrison:mirror} asserts that this should also hold for the K\"ahler cone  of a  Calabi-Yau manifold $X$ with $h^{2,0}(X)=0$. This too, has been verified in some cases. See also \cite{totaro}
\end{example}

\begin{question}
Given a pair of polyhedral type 
$(C_+,\Gamma )$, do $\Gamma$
and the cone generated by the
$\Gamma $-orbit of a rational point of
$C_+-\{ 0\}$ form a pair of polyhedral type?
\end{question}

There is in general  no subgroup $\Gamma$
of $\GL(V(\QQ))$ which forms with $C_+$ a pair
of polyhedral type. But if there is one, then the next result says that
all such subgroups belong to a single
commensurability class. (Recall that two subgroups of
some group are said to be \emph{commensurable} if their
intersection is of finite index in each of them,
and that this is an equivalence relation.) 

\begin{proposition}\label{prop:commens} 
Let $(C_+,\Gamma)$ be a pair of polyhedral type, and let $\Gamma '$ be a subgroup of
$\GL(V(\QQ))$ which stabilizes $C$. Then $(C_+,\Gamma ')$
is of polyhedral type if and only if $\Gamma '$ is
commensurable with $\Gamma $.
\end{proposition} 

\begin{proof} \emph{`If':}  
Let $L\subset V(\QQ)$ be a lattice 
stabilized by $\G $. If $\G '$ contains
$\G $ as a subgroup of finite index, then
$\G '\cdot L$ is contained in a finite union of
lattices, and hence generates a lattice $L'\subset
V(\QQ)$. Clearly, $\G '$ stabilizes $L'$.
It then follows from the definition \ref{prop:poltype}-$i$, that
$(C_+, \G ')$ is of polyhedral type. If on the
other hand $\Gamma '$ is a subgroup of finite index
of $\Gamma $, then choose a finite system $S\subset
\Gamma$ of representatives of left cosets of
$\Gamma '$ in $\Gamma $. If $\Pi$ is a \rp \
cone in $C_+$ such that $\Gamma \cdot \Pi =
C_+$, then $\Pi '=\sum_{s\in S} s(\Pi )$ is a \rp
\ cone satisfying $\Gamma '\cdot \Pi '=C_+$, and
so $(C_+,\Gamma )$ is in this case of polyhedral type,
too.

\emph{`Only if':}  
Assume that $(C_+,\Gamma )$ is of polyhedral
type. If $L'\subset V(\QQ)$ is a lattice
stabilized by $\Gamma '$, then $L'\supset kL$,
for some $k\in \NN$, and so $L\supset L'\cap
L\supset kL$.
Since $L/kL$ is finite, the group of $\gamma \in
\Gamma$ stabilizing $L'\cap L$ is of finite index in
$\Gamma$, and hence its action on $C_+$ is of 
polyhedral type. A similar assertion holds for the group of
$\gamma '\in \Gamma '$ which stabilize $L\cap L'$.
So without loss of generality we can assume that
$\Gamma $ and $\Gamma '$ both stabilize a lattice
$L\subset V(\QQ)$. It is enough to prove that 
the group $\Gamma ''$ of $\gamma \in \GL(V)$ which 
leave both $L$ and $C$ invariant, contains $\Gamma
$ and $\Gamma '$ as subgroups of finite index. Let
$\Pi$ be a \rp \ cone in $C_+$ such that $\Gamma
\cdot \Pi\supset C$, and let $S$ denote the set of
$\gamma ''\in \Gamma ''$ with $\gamma ''(\Pi )\cap
\Pi \cap C\not= \emptyset$. By the Siegel property
\ref{prop:siegel}, $S$ is finite. For every $\gamma ''\in \Gamma
''$ there exists a $\gamma \in \Gamma $ such that
$\gamma (\Pi )$ meets $\gamma ''(\Pi )\cap C$, so that
$\gamma ^{-1}\gamma ''\in S$. This proves that
$\Gamma ''=S\cdot \Gamma$ and hence that $\Gamma $
is of finite index in $\Gamma ''$. For the same
reason, $\Gamma '$ is of finite index in $\Gamma ''$.
\end{proof}

In the remainder of this section, 
$(C_+,\Gamma)$ is of 
polyhedral type and $L\subset V(\QQ)$
is a $\Gamma $-invariant lattice.

\begin{proposition}\label{prop:inherit}
Let $\Sigma $ be a $\Gamma $-invariant locally
\rp \ decomposition of $C_+$, and let $\sigma
\in \Sigma$. Then $(\sigma ,\Gamma (\sigma ))$ is of
polyhedral type, and if $F$ denotes the
smallest face of $C_+$ which contains $\sigma $,
then $\Star_{\Sigma}(\sigma )$ decomposes into a
finite number of $Z_{\Gamma }(F)$-equivalence
classes.
\end{proposition}  

\begin{proof}
Let $\Pi$ be a \rp \ cone in $C_+$ with $\G\cdot\Pi=C_+$. Then
$\Pi$ meets the relative interiors of only
finitely many members of $\Sigma $. Let $\gamma
_1,\dots ,\gamma _N \in \Gamma$ be such that
$\gamma _1(\mathring \sigma ),\dots ,\gamma
_N(\mathring \sigma )$ are the $\Gamma $-translates of
$\mathring \sigma $ which meet $\Pi $. Then $\Pi_1:=
(\gamma _1^{-1}(\Pi )+\cdots +\gamma _N^{-1}(\Pi ))\cap
\sigma $ is a \rp\ cone. For every $x\in \mathring \sigma$, 
there exists a $\gamma \in \Gamma $ such that
$\gamma (x) \in \Pi $. Then $\gamma (\mathring \sigma)\cap
\Pi \not= \emptyset$ and so $\gamma (\mathring \sigma )=
\gamma _{\nu }(\mathring \sigma )$ for some $\nu\in
\{1,\dots ,N\}$. This implies that $\gamma _{\nu
}^{-1}\gamma $ leaves $\sigma $ invariant and maps $x$
into $\Pi _1$. So $\Gamma (\sigma )\cdot \Pi_1
\supset \mathring \sigma $. As every \rp\ cone in $C_+$
intersects $\sigma $ in a \rp \ cone, we have $(\mathring
\sigma )_+=\sigma $. This proves the first assertion.

Next we fix a $x_0 \in\mathring\sigma\cap V(\QQ)$
which is not a fixed point of
a nonidentity element of $\Gamma (\sigma )$. We prove
that for every $\tau \in \Sigma $ with $\tau \supset
\sigma $, there exists a $\gamma _{\tau }\in \Gamma $ 
such that $\gamma _{\tau }(x_0) \in \Pi $ and 
$\gamma _{\tau } (\mathring \tau)\cap \Pi\not=\emptyset$. 
This will imply the last
assertion, for  $\Gamma x_0\cap \Pi $ and $\Sigma |\Pi$ 
are finite. To see that such
a $\gamma_{\tau }$
exists, choose a \rp \ cone $\Pi _{\tau
}\subset \tau $ with $\Pi _{\tau }\cap \mathring \tau
\not= \emptyset $ and $\Pi _{\tau }\cap \sigma =
\RR_{\geq 0}x_0$. Since $\{\Pi _{\tau }\cap
\gamma ^{-1}(\Pi )|\gamma \in \Gamma \}$ is a finite
collection of \rp \ cones which covers $\Pi _{\tau
}$, there exists a $\gamma _{\tau }\in \Gamma $ with
$x_0 \in \gamma _{\tau}^{-1}(\Pi )$ and $\mathring \Pi
_{\tau }\cap \gamma _{\tau }^{-1}(\Pi )\not=
\emptyset $. So $\gamma _{\tau }(x_0)\in \Pi $ and
$\gamma _{\tau }(\mathring \tau )\cap \Pi \not=
\emptyset $, as required.
\end{proof}

\begin{example}\label{example:arrangement}
 Here is an example of a nontrivial situation to which  
the previous proposition applies.
Let $\langle \, ,\, \rangle $ be a symmetric bilinear
form on $V$ of signature $(1,\dim V-1)$ defined over
$\QQ$, and let $C$ be a connected
component of the set of $x\in V$ with $\langle
x,x\rangle >0$. We choose a lattice $L$ in $V(\QQ)$,
and let $\Gamma :=\Orth (L)\cap \aut (C)$.
It follows from Proposition \ref{prop:groupexample} that 
$(V(\QQ),C,\Gamma )$ is of
polyhedral type. Suppose now further be
given a collection $\Hcal $ of hyperplanes of $V$
defined over $\QQ$ meeting $C$, which is a finite
union of $\Gamma $-orbits.

\smallskip 
\emph{Claim. } The collection of hyperplanes $\Hcal$ induces 
a $\Gamma $-invariant \lrp\ decomposition $\Sigma $ of $C_+$.

\begin{proof}
We must show that for every \rp\ cone $\Pi$ in
$C_+$, the collection $\{ H\cap \Pi | H\in \Hcal\}$
has only finitely distinct members. Given $H\in \Hcal$, 
then $C_+\cap H$ and the
group of $\gamma \in \Orth (L\cap H)$ which
preserve $C\cap H$ make up a pair of polyhedral type
(of one dimension lower, but otherwise of the
same type as $(C_+,\Gamma )$). It is
not hard to show that $\Gamma (H)$ is of finite
index in the latter group, and so by \ref{prop:commens},
$(C_+\cap H, \Gamma (H))$ is also of polyhedral
type. Hence there exists a \rp \ cone $\Pi _H\subset
C_+\cap H$ such that $\Gamma (H)\cdot \Pi _H=C_+\cap
H$. By the Siegel property \ref{prop:siegel}, the collection
$\{\gamma (\Pi _H)\cap \Pi |\gamma \in \Gamma \}$ has
finitely many distinct members. If we let $H$ run
over a representative system of $\Gamma
$-equivalence classes in $\Hcal$, we find that the
same is true for the collection $\{H\cap \Pi |H\in
\Hcal\}$.
\end{proof}
\end{example}

This construction often yields \lrp\ decompositions
of $C_+$ for which the adverb ``locally'' can not be
dropped, and thus produces in view of \ref{prop:inherit} 
also interesting new
examples of pairs of polyhedral type. For instance,
given a union $T$ of conjugacy
classes of reflections in $\Gamma $, then because  
such conjugacy classes are finite in number,
the collection $\Hcal$ of fixed point
hyperplanes of the members of $T$ breaks up in a
finite number of $\Gamma $-equivalence classes.
The resulting \lrp\ decomposition of $C_+$ is
\rp\ if and only if the subgroup of $\Gamma $
generated by $T$ is of finite index in $\Gamma $.
This follows from work of Vinberg \cite{vin}. But
according to this very author \cite{vin2}, in only
a few cases the subgroup of $\Gamma$ generated by
its reflections is of finite index in $\Gamma$. 

\begin{proposition}\label{prop:cores} 
Every $\G$-invariant kernel for $C$ contains
a core and is contained in a cocore for $C$. Moreover
$[(\bar C-\{ 0\} )\cap L]+\bar C$ is a cocore for $C$,
and dually, $((C^*-\{ 0\})\cap L^*)^\v$ is a core for $C$.
\end{proposition}

\begin{proof} Choose a \rp\ cone $\Pi\subset C_+$ such that
$\G\cdot\Pi =C_+$. Let $K$ be a $\Gamma$-invariant kernel for $C$. 
Since $\Pi\cap(C^\circ \cap L^*)^\v$
contains the convex hull of $(\Pi -\{ 0\})\cap L$ and
$0\notin \bar K\cap \Pi$, there exists a
$\lambda >0$ such that $\bar K\cap \Pi \subset \lambda
(C^\circ \cap L^*)^\v$. This implies that the last set also
contains $\bar K\cap C$. As $\bar K\cap C$ is dense in
$\bar K$, it follows that it even contains $\bar K$.
Applying this to $K^\v$, we also find that 
$\bar K=K^{\v\v}$ contains a set of the form $\mu [C\cap L]$ for
some $\mu >0$. Hence  $K$ contains $2\mu [C\cap L]$.

Let $\Lambda$ denote the set of lattice points in
$\bar C-\{ 0 \}$. Clearly, $(C^\circ \cap L^*)^\v$ 
contains $\Lambda$ and hence also $[\Lambda ]+\bar C$. 
If $\nu >0$ is such that $\Pi \cap
(C^\circ\cap L^*)^\v\subset \nu [(\Pi -\{0 \})\cap L]$, then
$C\cap (C^\circ \cap L^*)^\v
\subset \nu [\Lambda ]$. Following Lemma \ref{lemma:facial}, $[\Lambda
]+\bar C$ is closed, and since $C\cap (C^\circ
\cap L^*)^\v$ is dense in 
$(C^\circ\cap L^*)^\v $, it follows
that the last space is contained in $\nu [\Lambda
]+\bar C$. This proves that $[\Lambda ]+\bar C$ is a
cocore for $C$. If we apply this to the dual
situation and dualize, we find that $((C^*-\{ 0\})\cap
L^*)^\v $ is a core for $C$.
\end{proof}

\begin{remark}
It is in general not true that a cocore is contained in $C_+$. 
To see this, suppose that there exist a face $F\not=\{
0\}$ of $C_+$, and a proper face $G$ of $\bar C$ which
contains $F$ and whose relative interior does not
contain any rational point.
Then no cocore is contained in $C_+$: if
$L\subset V(\QQ)$ is a lattice, choose $p\in \mathring
F\cap L$, so that $p$ belongs to the
typical cocore $K:=(C^\circ\cap L^*)^\v$.
Hence $K\supset p+\bar C\supset p+\mathring G$, 
and the last space is a nonempty
open subset of $\mathring G$ which by
assumption does not meet $C_+$.

To be more concrete, let $V$ be the space of
symmetric bilinear forms on $\RR^n$,
$n\geq 3$ with its standard rational structure,
and let $C$ be the cone of positive definite
forms. (We are in a special case of \ref{prop:groupexample} if we
take $\Gcal :=\PSL_n$.) Choose an
irrational line $l$ in $\RR^{n-1}\subset 
\RR^n$, and let $F$ resp. $G$ be the cone of
positive semi-definite forms on $\RR^n$ whose
nilspace contains $\RR^{n-1}$ resp. $l$. Then
$F$ is the half line spanned by $x_n^2$ and is a
face of $C_+$, whereas $G$ is a face of
$\bar C$ which contains $F$, but has no
rational points in its relative interior.
\end{remark}

So while a $\Gamma$-invariant kernel need not be contained in $C_+$, the following proposition
shows that its intersection with $C_+$ is quite nice (compare Proposition II 5.22 of \cite{amrt}). 

\begin{proposition}\label{prop:lrpkernel} 
Let $K$ be a $\Gamma$-invariant kernel for $C$. 
Then the following are equivalent:
\begin{enumerate}
\item[$(i)$] $K$ is \lrp\ in $C_+$ .
\item[$(ii)$] There exists a finite union
$S$ of $\Gamma $-orbits in $(\bar C-\{ 0\})\cap V(\QQ)$ 
such that $K\cap C_+=[S]+C_+$.
\item[$(i)^*$] $K^\v$ is \lrp\ in $C^\circ_+$.
\item[$(ii)^*$] There exists a finite union
$S^*$ of $\Gamma $-orbits in $(C^*-\{ 0\})\cap V(\QQ)$ 
such that $K^\v\cap C^\circ_+=[S^*]+C^\circ_+$.
\end{enumerate}
Moreover, if one of these conditions is
fulfilled, then we can take for $S$ the set $E$ of extreme points of $K$ and we have 
$\bar K=[E]+\bar C$ and  every bounded face
of $K$ is a rational polyhedron.
\end{proposition} 

\begin{proof} 
$(i)\Rightarrow (ii)$ plus the last clause: Let
$\Pi \subset C_+$ be a \rp\ cone such that $\G
\cdot \Pi =C_+$. Since $K\cap \Pi$ is a rational
polyhedron, the set $E_0$ of its extreme points is a 
finite set of rational points with the property that
$K\cap \Pi = ([E_0]+C_+)\cap \Pi$. So
if we let $E:= \Gamma \cdot E_0$, then $K\cap C_+=[E]+C_+$
and $E$ is the set of extreme points of $K\cap C_+$. 

Let $k\in \NN$ be such that $E_0\subset\frac{1}{k} L$.
Then $E\subset \frac{1}{k}L$, which shows that $E$ is
discrete in $V$. According to Lemma \ref{lemma:facial}, this implies
that $[E]+\bar C$ is closed in $V$ and that every bounded face of
$[E]+\bar C$ is spanned by a finite subset of $S$ and is therefore a rational 
polyhedron. In particular, it is a face of $K\cap C_+$. Since $K\cap C_+=[S]+C_+$
is dense in $K$, we have $\bar K=[S]+\bar C$ and so
every bounded face of $K$ is also one of $\bar K$. Hence it is of the stated form.

$(ii)\Rightarrow (i)^*$: Argueing as above we find that
$S$ is contained in some lattice in $V(\QQ)$.
We then conclude from Lemma \ref{lemma:finitepol} that $([S]+\bar
C)^\v$ is \lrp \ in $C_+$.

The proposition now follows from the proven
implications and their dual versions.
\end{proof}

\begin{definition}
We call a function $f:C_+\to \RR$ 
{\it admissible} if $f$ is continuous and for every
\rp \ cone $\Pi \subset C_+$, the set of $(x,t)\in
\Pi \times\RR$ with $f(x)\geq t$ is a \rp \ cone. So 
\[
C_f:=\{(x,t)\in C\times \RR\, |\, f(x)>t\}
\]
is an open nondegenerate convex cone in $V\times \RR$ with 
$C_{f,+}=\{(x,t)\in C_+\times\RR\, |\, f(x)\geq t\}$. 
The interest of such a function lies
in the fact that it determines a decomposition 
$\Sigma (f)$ of $C_+$: the members of this
decomposition are simply the projections of the faces
of $C_f$ which do not contain the negative $t$-axis.
An alternative characterization of $\Sigma (f)$ is
that it is the coarsest \lrp \ decomposition of $C_+$ with
the property that $f$ is linear on each member.
\end{definition}

We return to Example \ref{example:arrangement} and prove that the
decomposition described there comes from an
admissible function. Fix a maximal
member $\sigma\in\Sigma $, and let for every
$H\in \Hcal$, $\xi_H$ be the unique indivisible
element of $L^*$, which defines $H$ and is $\geq 0$ on
$\sigma $. For $x\in C_+$, we define
\[
f(x)=\sum_{H\in \Hcal}\min\{ \xi _H(x),0\}.
\]
The sum involves at most a finite number of nonzero
terms, since at most finitely many $H\in \Hcal$
will separate $x$ from $\sigma$. It is easily
verified that $f$ is admissible and that $\Sigma
(f)=\Sigma $. Notice that $f$ transforms
under $\gamma\in\G$ as follows: $f\gamma ^{-1}= f+\sum_{H\in
\Hcal(\gamma )} \xi_H$, 
where $\Hcal(\gamma)$ denotes the collection of
$H\in \Hcal$ which separate $\gamma^{-1}(\sigma )$ from $\sigma $.
So $f$ is not $\Gamma $-invariant
(unless $\Hcal=\emptyset$), but $\gamma \mapsto
f\gamma ^{-1}-f$ is a 1-cocycle on $\Gamma$ with
values in the $\Gamma$-representation $V(\QQ)^*$ (in fact, even in $L^*$).  
But if this cocycle happens to be a coboundary, then by definition there exists an $\rho \in V^*(\QQ)$ such that  
$f-\rho$ is $\G$-invariant. Observe that this function is also admissible and defines the same 
decomposition as $f$. This phenomenon occurs in the theory of generalized root systems, where
$\rho$ appears as what some authors call a `Weyl vector'.

\medskip
Interesting examples of $\Gamma $-invariant
admissible functions (and hence of $\Gamma$-invariant
\lrp\ decompositions) are obtained from
$\Gamma$-invariant \lrp\ kernels:

\begin{lemma} 
Let K be a  \lrp\ kernel for
$C^\circ$, invariant under $\Gamma$. Then every $x\in C_+$ has a minimum  on
$K$, and if we denote this minimum by $f_K(x)$,
then $f_K$ is a $\Gamma$-invariant admissible
function on $C_+$ and $K^\v\cap C_+=\{x\in
C_+|f_K(x)\geq 1\}$.
\end{lemma}

\begin{proof} Let $E$ denote the set of extreme points of
$K$. By Proposition \ref{prop:lrpkernel}, $\bar K=[E]+\bar C$, and so it follows that for $x\in
C_+$, $\inf_K x=\inf_{\bar K} x=\inf_E x$. Write $x=\lambda_1x_1+\cdots
+\lambda_mx_m$ with $x_{\mu }$ a rational
point of $C_+$ and $\lambda_{\mu }\geq 0$. Since $E$
is contained in a lattice in $V(\QQ)^*$, $x_{\mu
}(E)$ will be a discrete subset of $\RR_{\geq 0}$, $\mu =1,\dots ,m$. 
Hence the same is true for $x(E)$. In particular, 
$x(E)$ has a minimum.

Now let $\Pi $ be a \rp\ cone in $C_+$. Then $\Pi
\cap K^\v$ is a rational polyhedron (which may
be empty). Let $\phi :\Pi \rightarrow \RR_{\geq
0}$ be the function characterized by $\phi (\lambda
x)=\lambda \phi (x)$, 
$x\in \Pi ,\lambda \in \RR_{\geq 0}$, and 
$\{x\in \Pi | \phi (x)\geq 1\}=\Pi
\cap K^\v$. Then
the set of $(x,t)\in \Pi \times \RR$ with $\phi
(x)\geq t$ is a \rp\ cone, and it is clear that 
$\phi =f_K|\Pi $. So $f_K$ is admissible and 
$K^\v\cap C_+=\{x\in C_+|f_K(x)\geq 1\}$.  
\end{proof}
    
Let us now consider the special case when $K$ is a
$\Gamma$-invariant \lrp\ core  for $C^\circ$. Then
$K^\v$ is a $\G$-invariant \lrp\ cocore for
$C$. 
For every \rp\ cone
$\Pi\subset C_+$, $K^\v\subset \Pi$ is a rational polyhedron
which meets each extremal ray of $\Pi$ (for $K^\v$ is
comparable with the standard cocore  $[C^\circ\cap L^*]^\v$,
which  has that property by Proposition \ref{prop:rpdec}). 
It is then easily seen that the bounded faces of $K^\v\subset \Pi$
lie on bounded faces of $K^\v$. So the cone spanned by the union of
the bounded faces of $K^\v$ coincides with $C_+$.

According to Proposition \ref{prop:lrpkernel} every bounded face of $K^\v$ is
a rational polyhedron, so that 
$\Sigma (K):=\Sigma (f_{K^\v})$
is in fact a \rp\ decomposition of $C_+$. Moreover,
the faces of $K$ parameterize in a bijective
manner the  faces of $\Sigma (f_{K^\v})$
by assigning to a face $P$ of $K$ the cone
$\sigma (P)$ of $x\in V$ with
the property that $x|K$ assumes its infimum
on all of $P$. In particular, $C_+=\cup \sigma (P)$
is the set of $x\in V$ which have a minimum on $K$.
Notice that $P\mapsto \sigma (P)$ reverses inclusions
and that $\dim P+\dim\sigma (P)=\dim V$. This generalizes
the construction of $\Sigma (C,L)$ of \ref{prop:rpdec}, for the
latter is obtained if we take $K=[C^{\circ }\cap
L^*]$.

\begin{application}[Construction of a polyhedral 
$\Gamma $-fundamental domain in $C_+$]\label{app:fundomain}
Choose $\xi\in C^\circ \cap V^*(\QQ)$. Then
it follows from Proposition \ref{prop:lrpkernel} that $K=[\Gamma \xi]+C^*$ is
a $\G$-invariant \lrp\ kernel. If $\lambda \in
\NN$ is such that $\lambda \xi \in L^*$, then
$\lambda K\subset [C^\circ\cap L^*]$ and so by
Proposition \ref{prop:cores} $K$ is a core for $C^\circ$. Every
extreme point  of $K$ corresponds to a maximal
element of $\Sigma (K)$. Since $\G\xi$ is the
set of extreme points of $K$, it follows that
$\G$ is transitive on the collection of
maximal members of $\Sigma (K)$. So
\[
\sigma :=\sigma (\{ \xi \} )=\{x\in C_+\, |\, \xi (\gamma
x)\geq \xi (x)\; \text{ for all }\; \gamma \in
\Gamma \}
\]
is a  \rp\ cone with the property that $\G\cdot \sigma =C_+$ 
and $\gamma (\sigma )\cap \mathring\sigma = \emptyset$ if 
$\gamma \in \Gamma$ does not fix $\xi$. In particular, 
$\sigma $ is a fundamental domain in $C_+$ if 
$\G_{\xi }=\{1 \}$. It also follows that
$C_+$ is just the set of $x\in \bar C$ which
have a minimum on $\G x$.

If we take $\xi \in \mathring F\cap V^*(\QQ)$, where
$F$ is a proper face of $C^{\circ }_+ -\{0\}$, then the
corresponding decomposition $\Sigma (K)$ is
also of interest. Again, $\G$ is then
transitive on the maximal members of $\Sigma (K)$,
and the stabilizer of $\sigma (\{ \xi \} )$ (which
is one such member) is $\Gamma _{\xi}$. Notice
that $\Gamma _{\xi }$ contains $Z_{\Gamma }(F)$ as
a subgroup of finite index; in general this will
be an infinite group.
\end{application}

We finally mention two consequences of having a polyhedral fundamental domain.

\begin{corollary}
The group $\G$ is finitely presented.
\end{corollary}

\begin{proof} 
Let $\Pi$ be a \rp\ cone in $C_+$ such
that $\Gamma\cdot\Pi =C_+$. So $\G\cdot (\Pi
\cap C)=C$. As is well-known (and easy to prove), 
the mere fact that $C$ is connected now implies that
$\G$ is generated by the $\gamma \in \G$ 
for which $\gamma (\Pi)\cap \Pi$ is a codimension one face of $\Pi$ which meets $C$. This is clearly a finite set.  
Similarly, the fact that $C$ is simply connected implies that
a complete set of relations among these generators is indexed by the codimension two faces 
of $\Pi$ which meet $C$.
\end{proof}

The other consequence involves the property VFL for groups (which stands for having \emph{Virtuellement  une r\'esolution Libre de type Finie}). Recall that a (discrete) group $G$ is said to be VFL of dimension $\le d$  if there exists a subgroup $H\subset G$ of finite index such that the trivial $H$-module $\ZZ$ admits a resolution  of length  $\le d$ by free finite rank $\ZZ[H]$-modules.

We show that $\G$  has this property. This is based on a standard construction, which we briefly recall. Consider the  $\G$-invariant decomposition $\Sigma$ of $C_+$ constructed  from the $\G$-stable lattice $L\subset V(\QQ)$ in \ref{app:fundomain}.
It has a  canonical ``barycentric'' subdivision  defined as follows: every member $\sigma\in\Sigma$, being a rational polyhedral cone,  has finitely many extremal rays. The sum of the integral generators of these rays spans a ray $R_\sigma$ with $\mathring R_\sigma\subset \mathring\sigma$. Now $\Sigma$ is naturally refined by a decomposition $\Sigma'$ whose members
$\not=\{ 0\}$ are indexed by the strictly monotonous sequences $\sigma_\pt:= (\sigma_0\supsetneq\sigma_1\supsetneq\cdots \supsetneq \sigma_k\not= 0)$ in $\Sigma$, the associated polyhedral cone being $\la \sigma_\pt\ra:=R_{\sigma_0}+\cdots +R_{\sigma_k}$. Notice that if $\sigma_k$ meets $C$, then $\la \sigma_\pt\ra \subset C\cup\{ 0\}$ so that its projectivization $P(\la \sigma \ra)$ is a polyhedron entirely contained in the open contractible $P(C)\subset P(V)$. If we denote by $\Sigma'_f$ the subcollection of $\sigma_\pt\in\Sigma'$ with that property, then  the union $P|\Sigma'_f|\subset P|\Sigma|$ of such polyhedra (often called the \emph{spine} of $\Sigma$) is 
a polyhedral complex of dimension $\le \dim V-1$ that is  invariant under $\G$ and whose polyhedral cells decompose into finitely many $\G$-equivalence classes.

There is a natural $\G$-equivariant deformation retraction of $P(C)$  onto this spine as follows. 
If $\sigma_\pt=(\sigma_0\supsetneq\sigma_1\supsetneq\cdots \supsetneq \sigma_k\not= 0)$ is any  strictly monotonous sequence in $\Sigma$ with $\sigma_0\cap C\not=\emptyset$,  then let $r\in\{ 0,\dots ,k\}$ be the highest index for which $\sigma_r$ still meets $C$ and denote by $\sigma_\pt^C$ the truncation  $\sigma_0\supsetneq\cdots \supsetneq \sigma_r$. There is a natural deformation retraction of the improper  polyhedron $P(\la\sigma_\pt\ra)\cap P(C)$ onto spinal polyhedron
 $P(\la\sigma^C_\pt\ra)$. It is compatible with inclusion and so this results in a $\G$-equivariant deformation retraction of $P(C)$ onto the spine $P| \Sigma'_f|$. In particular, $P| \Sigma'_f|$  is contractible.

\begin{corollary}\label{cor:vfl}
The group $\G$ is  \emph{of type VFL of dimension} $\le \dim V-1$.
\end{corollary}
\begin{proof}
Denote by $\G'$ the kernel of the representation of 
$\G$ on $L/3L$. A well-known theorem of Serre asserts that $\G'$ is torsion free. So $\G'$ acts freely  on $P(C)$  and hence also on $P| \Sigma'_f|$. Upon replacing $\G'$ by a smaller subgroup (still of finite index in $\G$) we may assume that the $\G'$-stabilizer of any $\sigma\in\Sigma$ which meets $C$ is trivial. The result is that the cells of $P| \Sigma'_f|$ have the same property. The associated chain complex therefore provides a  resolution of of length   $\le \dim V-1$ of the trivial $\G'$-module $\ZZ$ by free finite rank $\ZZ[\G']$-modules. 
\end{proof}

\begin{question}
Corollary \ref{cor:vfl} implies among other things that  the cohomology of $\G$ with values in a finite dimensional 
$\QQ$-vector space that is also a representation of $\G$ is finite dimensional. A case of particular interest is $H^1(\G, V)$ (which has a $\QQ$-structure for which $H^1(\G, V)(\QQ)=H^1(\G, V(\QQ))$). Any $c\in H^1(\G, V)$ is representable by a cocycle, i.e., a map $\g\in\G\mapsto c_\g\in V$ satisfying   
$c_{\g_1\g_2}=c_{\g_1} +\g_1(c_{\g_2})$. This defines an action of $\G$ on a copy $V_c$ of $V$
by affine-linear transformations defined by the rule $\g_c(v):= c_\g +\g (v)$ (so its linear part is the given action). The $\G$-action on $V_c$ has  fixed  point if and only if the class $c$ is zero.
We can make it depend linearly on $c$ by choosing representative cocyles  for a basis
of $H^1(\G, V)$ and then extending linearly the resulting actions. This yields an exact sequence of $\G$-representations
\[
0\to V\to\tilde V\to  H^1(\G, V)\to 0,
\]
where $\G$ acts of course trivially on $H^1(\G, V)$. It is  universal for that property
in a sense we don't bother to make precise. If we choose the basis in $H^1(\G, V)(\QQ)$ and let the representative cocycles take their values in $V(\QQ)$, then $\tilde V$ acquires a $\QQ$-structure  preserved by $\G$. We can even do better and take the basis in the image of $H^1(\G, L)\to H^1(G,V)$, let the representative cocycles take their values in $L$ and get a lattice $\tilde L$ in $\tilde V(\QQ)$ preserved by $\G$.

Does there exist an open nondegenerate convex cone $\tilde C$ in $\tilde V$ which forms with $\G$ a pair of polyhedral type and is such that $\tilde C_+$ contains $C_+$ as a face?
\end{question}

\section{The Stabilizer of a Face}

Throughout this section, we fix a polyhedral  triple 
$(V(\QQ),C,\G)$ in the sense
of Proposition \ref{prop:poltype} and a face $F$ of $C_+$. 
Our principal goal is to describe the structure of the 
$\G$-stabilizer of $F$.  

\medskip
We begin with a bit of notation. We let $F^\d$ stand
for the set of $\xi \in C^0_+$ which vanish on $F$.
This is clearly an exposed face of $C^0_+$ and
its annihilator contains $F$. We shall find that $F^{\d\d}=F$, 
but at this point it is not even  clear whether $F\not=C_+$ implies  $F^\d\not=\{ 0\}$.
We denote the linear span of $F$  in $V$ by $V_F$ and  write $V^F$ for
the annihilator of $F^\d$ (it will turn out that there is no conflict with
that same notation  used in Section 1). So we have a flag of $\QQ$ vector spaces defined over $\QQ$:
\[
0\subset V_F\subset V^F\subset V.
\]
We further put $T_F:= V^F/V_F$ and denote the projections
\[
\pi_F: V\to V/V_F,\quad \pi^F: V\to V/V^F,
\]
so that the latter is the composite of $\pi_F$ and the projection
\[
q_F: V/V_F\to V/V^F.
\]
Observe that we have a perfect duality $V/V^F\times V^*_{F^\d}\to \RR$ and that
under this duality  $\pi^F(C)$ is identified with the open dual 
of $F^\d$. It is in particular a nondegenerate convex cone.  
Let us begin with stating one of the main results of this section.
Denote by $N_\G(F)$ the $\G$-stabilizer of $F$. It acts on $F$ and $F^\d$ and so we have a
group homomorphism $N_\G(F)\to \G (F)\times \G (F^\d)$. 

The following theorem sums up the main content of this section.

\begin{theorem}\label{thm:stabilizerfacemain}
The image of the projection $N_\G(F)\to \G (F)\times \G (F^\d)$ is of finite index in the latter and the elements in  its  kernel  that act trivially on $T_F$ form a free abelian subgroup $U_\G(F)$ of finite index in that kernel. The action of $U_\G(F)$ on $V$ is $2$-step unipotent and is given by a unique homomorphism
\[
u\in U_\G(F)\mapsto \sigma_u\in \Hom (V/V_F,V^F)
\]
with the following properties:
\begin{enumerate} 
\item[($i$)] $\sigma_u$ maps $T_F$ to $V_F$ and the induced maps 
\[
j_u: V/V^F\to T_F \text{ resp. } k_u: T_F\to V_F
\]
are such that $u(x)=x+\sigma_u(x')+\frac{1}{2}k_uj_u (x'')$, where 
$x'$ resp.\ $x''$ denote the images of $x$ in $V/V_F$
resp.\ $V/V^F$,
\item[($ii$)] for $u,v\in U_\G(F)$, we have $k_uj_v=k_vj_u$,
\item[($iii$)] $k_uj_u$ maps $\pi^F(C)$ to $F-\{ 0\}$, unless $u=1$.
\end{enumerate} 
\end{theorem}

\begin{example}
This theorem is well illustrated by the following basic example. Take for $V$ the space $\sym^2W$ of symmetric tensors in $W\otimes W$, where $W$ is a real finite dimensional vector space with a $\QQ$-structure  and let $C\subset V$ be the cone of positive ones. Then $C_+$  is the cone spanned by  the pure squares $w\otimes w$ with $w\in W(\QQ)$. Alternatively, $C_+$ consists of the semipositive symmetric tensors whose annihilator is defined over $\QQ$. So a face  $F$ of  $C_+$ is given by a subspace $W'\subset W$ defined over $\QQ$ and then consists of the semipositive elements in $\sym^2 W'$. We have $V_F=\sym^2W'$, 
$V^F=W'\circ W$ (i.e., the span of the tensors $w'\otimes w+w\otimes w'$, with $w'\in W'$), so that 
$T_F=V^F/V_F$ may be identified with $(W/W')\otimes W'$ and $V/V^F$ with $\sym^2(W/W')$.
The open dual $C^\circ$ is the cone of positive definite quadratic forms on $W$ and under this identification, the relative interior $F^\d$ may be identified  with the cone of positive definite quadratic forms on $W/W'$. The group $\GL (W)$ acts on $(V,C)$ and the stabilizer of $F$ is the  stabilizer of $W'$. The latter maps onto $\GL(W')\times \GL (W/W')$ (its Levi quotient) with kernel an abelian unipotent group $U (F)$ that can be identified with the vector group 
$\Hom (W/W', W')$. The map $\sigma : U (F)\to \Hom (V/V_F,V^F)$ is identified with the map 
\[
\Hom (W/W', W')\to \Hom (\sym^2W/\sym^2 W', W'\circ W),
\]
which assigns to $u\in \Hom (W/W', W')$ the map $\sym^2(W/W')\to W'\circ W$ characterized by
$w\otimes w+\sym^2W'\mapsto u(\bar w)\otimes w +w\otimes u(\bar w)$ (here $w\in W$ and $\bar w$ is its image in $W/W'$). Notice that this induces maps
\begin{align*}
j_u: \sym^2 (W/W')\to (W/W')\otimes W',\quad  & \bar w\otimes\bar w\mapsto \bar w\otimes u(\bar w)\\
 k_u:   (W/W')\otimes W'\to \sym^2W',\quad  & \bar w\otimes k\mapsto u(\bar w)\otimes k+ k\otimes u(\bar w).
\end{align*}
so that $k_uj_v=u\otimes v+v\otimes u: \sym^2 (W/W')\to \sym^2W'$. We note that if $(w_i)_i$ is a basis of $W$, then $\half k_uj_u=u\otimes u$ sends $\sum_i w_i\otimes w_i$ to $\sum_i u(w_i)\otimes u(w_i)$, which is zero only when $u=0$.

If $\G\subset \SL (W)$ is arithmetic, then we have a similar description for $\G$-stabilizer of $F$ (which is of course 
the  $\G$-stabilizer of $W'$). 
\end{example}

We shall denote the kernel of  $N_\G(F)\to \G (F)\times \G (F^\d)$ by 
$Z_\G(F\times F^\d)$ (this is in agreement with our notational convention if we let $\G$ act on $V\times V^*$ diagonally). Furthermore, $L$ stands for some $\G$-invariant lattice in $V(\QQ)$.

\begin{lemma}\label{lemma:stucturefaces} 
Let $\Sigma$ be a $\G$-invariant \lrp\ 
decomposition of $C_+$, and let $\sigma \in
\Sigma$ be such that $\mathring\sigma$ is open in $F$.
Then every point of $\pi_F(C_+)$ is in the
$\pi_F$-image of the relative interior of a unique 
member of $\Star_{\Sigma}(\sigma)$, and the projection $\pi_F$  
maps the members of $\Star_{\Sigma}(\sigma)$ onto a \lrp\ decomposition
$\pi_{F\, *}\Star_{\Sigma}(\sigma)$ of $\pi_F(C_+)$.

If $\Sigma$ is in fact \rp\  and  $P$ is a \rp\ cone in $\pi^F(C_+)$ whose  preimage 
in $\pi_F(C_+)$ is denoted $\tilde P$, then 
the restriction of $\pi^\d_{F\, *}\Star_{\Sigma}(\sigma)$ to 
$\tilde P$ has only finitely many $Z_\G(F\times F^\d)$-orbits.
\end{lemma}

\begin{proof}
Given $x\in C_+$, choose a \rp\ cone $\Pi$ in $C_+$ which intersects 
$\mathring \sigma $ and contains $x$. Since $\Sigma |\Pi $ is a
finite decomposition into \rp\ cones, there is a 
$y\in \Pi\cap\mathring\sigma $ such that
$x+y$ is in the relative interior of a 
member of $\Star_{\Sigma}(\sigma)$. This member is
clearly independent of the choice of $y$.
The first part of the lemma now follows easily.

For last clause of the lemma we may assume that $\mathring P\subset \pi^F(C)$. It
suffices to show that the collection of 
$\tau \in \Star_{\Sigma }(\sigma )$ 
whose image in $\pi^F(C_+)$ meets $\tilde P$ is finite
modulo $Z_\G(F\times F^\d)$. But this follows from the fact that the
collection $\Star_{\Sigma}(\sigma)$ is finite modulo
$Z_\G(F)$ by Proposition \ref{prop:inherit} and the Siegel property of
the image of the latter group in $\pi^F(C)_+$.  
\end{proof}

\begin{corollary}\label{cor:faceimages}
We have  $\pi^F(C_+)=\pi^F(C)_+$ 
and the image of $Z_\G (F)$ in $\G(F^\d)$ is a subgroup of the 
latter of finite index (or equivalently, the image of
$N_{\Gamma }(F)$ in $\G(F)\times \G(F^\d)$ is a subgroup 
of finite index).
\end{corollary}
\begin{proof}
Choose a $\G$-invariant \rp\ decomposition $\Sigma $.
It follows from Proposition \ref{prop:inherit} that there exists 
a rational polyhedral cone $\Pi$ in $C_+$ whose 
$Z_\G(F)$-orbit contains $|\Star_{\Sigma}(\sigma)|$.
Then $(i)$ of Lemma \ref{lemma:stucturefaces} implies that 
$Z_\G(F)\cdot \pi^F(\Pi)=\pi^F(C_+)$. The corollary now 
follows from \ref{prop:poltype}.
\end{proof}

\begin{remark} In contrast to first assertion of the
above Corollary, it may happen that $\pi_F(C)_+$ is
strictly greater than $\pi_F(C_+)$.
\end{remark}

\begin{proposition}\label{prop:facestructures} 
Let $G$ be a face of $C_+$ which contains $F$.
\begin{enumerate} 
\item[$(i)$] The common zero set of the set of 
rational linear forms on $V_G$ which
are $\geq 0$ on $G$ and vanish on $F$ is 
$V^F\cap V_G$.
\item[$(ii)$] The assignment $F\mapsto F^\d$ sets up bijection
between the faces of $C_+$ and those of
$C^\circ_+$ which reverses the inclusion relation. In particular,
$F^{\d\d}=F$ and $F$ is an exposed face of $C_+$.
\item[$(iii)$] The projections $\pi_F$ resp.\ $\pi^F$ 
map $G$ onto a face of $\pi_F(C_+)$ resp.\ 
$\pi^F(C_+)$ respectively; this sets up a
bijection between the collection of faces of $C_+$ which
contain $F$, the collection of faces of $\pi_F(C_+)$, and 
the collection of faces of $\pi^F(C_+)$.
\item[$(iv)$] The dual of $\pi^F(\mathring G)$
is naturally identified with the
closure of $\pi^\d_{G^\d}(F^\d)$.
\end{enumerate} 
\end{proposition} 
\begin{proof} We first prove $(i)$ under
the additional hypothesis that $G=G^{\d\d}$. (This 
assumption becomes superfluous once
we have proved $(ii)$.) We have to show that every
rational linear form $\xi $ on $V_G$ which is $\geq
0$ on $G$ extends to a rational linear form on $V$
which is $\geq 0$ on $C$. Since $G=G^{\d\d}$,
this follows from Corollary \ref{cor:faceimages} applied to 
$G^\d$: such $\xi$ a lies in  $\pi^G C^\circ_+)$. 

We next prove a special case of $(ii)$: We claim that if
$F^\d =\{ 0\} $, then $F=C_+$. Choose $x\in \mathring F\cap L$. Then for every
nonzero integral $\xi\in C^*$ we have $\xi (x)
\geq 1$ (otherwise $F^\d \not= \{ 0 \}$) and hence
$x\in [(C^*-\{ 0\} )\cap L^*]^\v$. According to Proposition
\ref{prop:cores} the last set is a core for $C$ and hence contained
in $C$. So $x\in C$ and hence $F=C_+$.

Now we prove $(ii)$ in general.  
Clearly, $F^{\d\d}=V^F\cap C_+\supset F$. We
may apply the above to $G:=F^{\d\d}$ and find
that there is no rational linear form on $V_G$ which is $\geq
0$ on $G$ and zero on $F$. Then $F=G$ by
the special case.

$(iii)$ Consider the face of $\pi^F(C_+)$ whose
relative interior contains $\pi^F(G)$. Its
preimage $H$ in $C_+$ is then a face  of
$C_+$ whose relative interior intersects 
$(V_G +V^F)\cap C_+$, and therefore also 
$V^G\cap C_+$ (for we have $V^G\supset V^F$). This last set is equal to $G$
(by $(ii)$). So $H=G$ and hence $\pi^F(G)$ is a face of   $\pi^F(C_+)$.
The assertion follows from this.

$(iv)$ By $(ii)$, $\pi^{G^\d}(C^{\circ })$ can be
regarded as the open dual of $\mathring G$. Then applying
$(ii)$ once more to the face $F$ of $G$ shows that
$\pi^F(\mathring G)$ can be
identified with the open dual of 
$\pi^{G^\d}(C^{\circ }_+) \cap \ann (F)$. By
$(iii)$, this last set is just $\pi^{G^\d}(F^\d)$.  
\end{proof}

\begin{corollary} 
Every rational linear form on the linear span of $F$ which is $\geq
0$ on $F$ extends to a rational
linear form on $V$ which is $\geq 0$ on $C$.
\end{corollary} 
\begin{proof}
This follows from the fact that 
$\pi^\d_{F^\d}$ maps  $C^\circ_+$ onto $(\pi^{F^\d}C^{\circ })_+$
(by Corollary \ref{cor:faceimages} and the fact that 
$\pi^{F^\d}C^\circ$ can be identified with 
the open dual of $\mathring F$ (by Proposition
\ref{prop:facestructures}-iii).    
\end{proof}

We shall need the following proposition.
 
\begin{proposition}\label{prop:Dbounded}
Let be given a real affine space $A$ of finite dimension, 
an affine lattice $A_\ZZ\subset A$,  an open convex subset $D$ of $A$ and
a group $\Delta$ of affine-linear transformations of $A$ which leave 
both $A_\ZZ$ and $D$ invariant. Assume that $\Delta$ has only a finitely many 
orbits in $A_\ZZ\cap D$. Then the asymptotic space of $D$ coincides with its 
recession cone: $\as (D)=\Tr (D)$ (i.e., $D +\as (D)=D$), $\as(D)$ is defined over $\QQ$,  
and  $\Delta$  acts on the affine space $A/\as(D)$ via a finite quotient.
\end{proposition}
  
\begin{proof} 
If $\as (D)=\{ 0\}$, then $D$ is bounded and there is nothing to show. 
We therefore assume that $\as (D)\not=\{ 0\}$. Then $\Tr (D)\not=0$  (\cite{rock}, Thm. 8.4).
We first show that $\Tr (D)$ is linear space.
If that is not the case, then let $R\subset\Tr (D)$ be ray such that line spanned by it is not contained in $\Tr (D)$. Choose an open ball $B\subset D$. Then  for any integer $n>0$, there exists in $B+R$ an interval $[x_n,y_n]$ whose end points lie in $A_\ZZ$ and which contains at least $n+1$ lattice points. Denote by $\Phi$ the collection of  affine linear maps $A\to \RR$ that are  integral on $A_\ZZ$, whose linear part is positive on $\Tr (D)-\as (D)$ and whose minimum on $D\cap A_\ZZ$ is $0$. This is a nonempty $\Delta$-invariant set and so  if $x\in D\cap A_\ZZ$, then the nonnegative integer  $\min_{f\in \Phi} f(x)$ only depends on the orbit $\Delta x$. We write $m(\Delta x)$ for this number. Now  for every $f\in \Phi$, $f(y_n)\ge  n+f(x_n)\ge  n$ and so $m(\Delta y_n)\ge n$. This contradicts the fact that $\Delta$ has finitely many orbits in $D\cap A_\ZZ$. 

The same argument shows that $\Tr (D)$ is defined over $\QQ$. 
If  $\pi : A\to A/\Tr (D)$ is the projection, then  $D=\pi^{-1}\pi D$, and so we must have $\Tr (\pi D)=\{ 0\}$. This implies that $\pi (D)$ is bounded, in other words $\as (D)\subset \Tr (D)$. The opposite inclusion is clear.
\end{proof}

\begin{corollary}\label{cor:Dall}
Suppose that in the situation of the previous proposition 
$\Delta$ acts on $D$ with compact fundamental domain. Then $D=A$.
\end{corollary}

\begin{proof} 
Following Proposition \ref{prop:Dbounded}, $\as(D)$ is a subspace defined over $\QQ$, 
$D+\as (D)=D$ and $\Delta$ acts on $A/\as (D)$  via a finite group. As it acts with compact 
fundamental domain on $D/\as (D)$, it follows that $D/\as (D)$ is compact. But $D/\as (D)$ is open in
the affine space $A/\as (D)$, and so this can only happen if $D/\as (D)$ is a singleton, i.e. if $D=A$.
\end{proof} 

\begin{corollary}\label{cor:Fdagger}
The closure of $F^\d$ in $V^*$ is just the set of
$\xi \in C^*$ which vanish on $F$ (and hence is an exposed 
face of $C^*$), and $\pi_F(C)$ is invariant under the 
translations in $T_F$.
\end{corollary} 

\begin{proof} 
Let $\tilde A$ be an affine subspace of $V$ parallel to $V^F$ 
which is defined over $\QQ$ and meets $C$. We let 
denote the images of $\tilde A$, $L\cap\tilde A$ and 
$C\cap\tilde A$  in $V/V_F$ by $A$, $A_\ZZ$ and $D$ respectively. 
It follows from Proposition \ref{prop:facestructures} 
that $D$ is also the image of  $C_+\cap\tilde A$. If $\Sigma$ and $\sigma$ are chosen as in Lemma \ref{lemma:stucturefaces},
then according to that lemma the restriction of $\pi_F(\Sigma _\sigma )$ 
to $D$ is a decomposition into compact rational polyhedra which is 
finite modulo $Z_\G(F\times F^\d)$. So  corollary \ref{cor:Dall} applies and we find that $D=A$.
This proves that $\pi_F(C)$ is invariant under the  translations in $T_F$.

The set of $\xi \in C^*$ which vanish on $F$ is an
exposed face of $C^*$ which contains $F^\d$.
Any such $\xi $ can be regarded as a linear form on
$V/V_F$ which is nonnegative on $\pi_FC$.
Since $\pi_FC$ is invariant under translations in $T_F$, it
follows that $\xi$ vanishes on $V^F$. So $\xi$ is in
the linear span of $F^\d$. The latter intersects $C^*$ in the closure of 
$F^\d$, and thus the corollary follows.
\end{proof}

\begin{lemma}\label{lemma:unip} 
The unipotent
elements in $Z_\G(F\times F^\d)$ form a normal
subgroup $U_\G (F)$ of finite index.
\end{lemma}

\begin{proof} 
We first prove that the characteristic polynomial of 
any $\gamma\in Z_\G(F\times F^\d)$ is a product of cyclotomic polynomials. 
This suffices: since there are only finitely
many such polynomials of given degree, 
it follows that the set of eigenvalues of elements
of $Z_\G(F\times F^\d)$ is finite. Now choose a strictly
increasing (Jordan-H\"older) filtration
$0=W_0\subset W_1\subset \ldots \subset V(\CC)$
invariant under $Z_\G(F\times F^\d)$ such that the image $G_i$ of
$Z_\G(F\times F^\d)$ in $GL(W_i/W_{i-1})$ is irreducible.
Clearly, the set of traces of elements of $G_i$ is
finite and a well-known  fact of representation theory
(see for instance \cite{curtisreiner}, proof of Burnside's
theorem (36.1)) then implies that $G_i$ is
finite. Hence the group of $\gamma\in Z_\G(F\times F^\d)$ that act
trivially on the quotients $G_i/G_{i-1}$ is of finite
index in $\Gamma$ and coincides
with the set of its unipotent elements. (This
argument was pointed out to me by O.~Gabber.)

The characteristic polynomial of $\g\in  Z_\G(F\times F^\d)$
has integral coefficients and so will be a product of cyclotomic 
polynomials once we show that every eigenvalue of
$\gamma$ has absolute value one. Suppose this
is not so: let $m>1$ the maximal absolute value that 
occurs and denote by $W$ the corresponding eigenspace
of $\gamma$ in $V$. Since $\gamma$ acts
trivially on $V/V^F$, we have $W\subset V^F$. 
Now choose a half line in $C$ 
which is not contained
in a proper eigenspace of $\gamma$. Then the
translates of this half line under the positive powers
of  $\gamma$ have a limiting half line
contained in $W\cap \bar C$, and hence contained in
$V^F\cap \bar C$. According to Corollary \ref{cor:Fdagger}, this last
intersection equals the closure of $F$ in
$V$. So $W\cap V_F \not= \{0\}$. But this is
impossible as $\gamma$ leaves $V_F$ pointwise
fixed.
\end{proof}

It follows from Corollary \ref{cor:faceimages} and Lemma \ref{lemma:unip}
that `up to finite groups' $N_\G (F)$ is an extension of $\G(F)\times \G(F^\d)$ by $U_\G(F)$. 
We shall now concentrate on the action of the latter on
$V$. We will find among other things that this group 
is abelian. 

Most of our information is obtained via the following proposition. In this proposition we regard the space of rays in a vector space as the boundary (the sphere at infinity) of any affine space over that vector space. 

\begin{proposition}\label{prop:unipotentgeneral}
Let $A$ be a real affine space of finite dimension (with translation space denoted $T$), 
$T_0\subset T$ a linear subspace,  $C_0\subset T_0$ a closed nondegenerate convex cone and
$U$ a unipotent group of affine-linear transformations of $A$ which leaves 
$T_0$ pointwise fixed. We assume that (i) that $A/T_0$ is spanned by some $U$-orbit and (ii) that
there exists a nonempty open subset  $D\subset A$ with the property that for every $a\in D$ and $1\not=u\in U$, the 
rays $\{\RR_{\ge 0}(u^k(a)-a)\}_{u\in U}$ have a  limiting 
ray in $C_0$. Put $A':=A/T_0$ and $T':=T/T_0$  and 
$a\in A\mapsto a'\in A'$ resp. $t\in T\mapsto t'\in T'$ denote the obvious projections.

Then  $U$ acts faithfully on $A'$ as a group of translations which spans $T'$; in particular, $U$ acts trivially on $T'$. 
Moreover, there exists a unique map  
\[
\sigma :A'\times T'\to T
\]
 with the following properties.
\begin{enumerate}
\item[a)] $\sigma$ is affine-linear in the first variable and for every
$a\in A$, $\sigma_{a'}:T'\rightarrow T$ is a linear section of the projection $T\to T'$ 
and so $\sigma$ induces a bilinear map 
\[
d\sigma :T'\times T'\rightarrow T_0
\]
characterized 
by the property that for $a'\in A'$ and $t'_1,t'_2\in T'$, $\sigma (a'+t'_1,t'_2)=
\sigma (a',t'_1)+d\sigma (t'_1,t'_2)$.
\item[b)] $d\sigma$ is a $C_0$-\emph{positive} symmetric form in the sense that it is 
symmetric, and for every nonzero $t'\in T'$, we have $d\sigma (t',t')\in C_0 -\{ 0\} $, and  
\item[c)] if  $u\in U$ is identified with $[u]\in T'$, then for all $a\in A$,
\[
u(a)=a+\sigma (\pi (a),[u])+\tfrac{1}{2}d\sigma ([u],[u]).
\]
\end{enumerate}
\end{proposition}

\begin{proof}

We use induction on $\dim (T')$.
To start the induction, assume $T'=\{ 0\}$. 
Then $U$ must act on $A$ as a group of translations. 
As $U$ preserves $D$, it follows that $U=\{ 1\}$ and we are done. 

From now on we assume $T'\not= \{ 0\}$ and $U\not= \{ 1\}$. Then $T_0\not= \{ 0\}$, for
the orbit of  a unipotent transformation  is either a singleton 
or has a limiting point at infinity.

Since $U$ is unipotent, we can find a $U$-invariant hyperplane $T_1$ of $T$
containing $T_0$. Then $U$ acts trivially on $T/T_1$ and hence acts on $A/T_1$
as a group of translations.   We denote the ensuing homomorphism 
$U\rightarrow T/T_1$ by $\alpha$ and write $U_1$ for its kernel. 
Choose a $T_1$-orbit $A_1$ in $A$ which intersects $D$. Clearly $U_1$ leaves
$A_1$ invariant and one verifies easily that the triple $(A_1,D\cap A_1,U_1)$
fulfills the  hypotheses of the proposition. So by induction $U_1$ acts
faithfully on $A'_1:=A_1/T_0$ as its full group of translations.

We can now prove the first assertion. 
Choose $a\in D$, and put
$e_i:=(u-1)^i(a)$. Then $e_1\in T$, $e_2\in T_1$,
$e_3\in T_0$, and $e_i=0$ for $i\geq 4$. So
\[
u^k(a)=a+\binom{k}{1}e_1+\binom{k}{2}e_2+\binom{k}{3}e_3
\]
for all $k\in \ZZ$. If $e_3\not= 0$, then the rays $\{ \RR_{\ge 0}(u^k(a)-a)\}_{k\ge 0}$ resp.\
$\{ \RR_{\ge 0}(u^{-k}(a)-a)\}_{k\ge 0}$ converge to $\RR_{\ge 0}e_3$ resp.\ $\RR_{\le 0}e_3$
and so $C_0$ would contain $\RR e_3$. This contradicts the nondegeneracy of $C_0$.
So $e_3=0$.  By a similar argument it follows that $e_2\in C_0$.  Since this is true for all $a\in D$, 
it follows that $u$ induces a translation in $A'$. If this translation is
trivial, then $e_1\in T_0$ and $e_2=0$. But then $C_0$ contains both $e_1$ and $-e_1$, and since
$C_0$ is nondegenerate, this can only happen when $e_1=0$, i.e., when $u=1$. 
This proves that $U$ acts faithfully on $A'$ as a translation group.  Since $A'$ is spanned by some $U$-orbit, this 
translation group must span $T'$.

To prove the remaining assertions, fix $a_0\in A$, and a linear section
$s:T'\to T$ of $T\to T'$. Then $A$ is
parameterized by
\[
(t',t_0)\in T'\times T_0\mapsto a_0+s(t')+t_0\in A.
\] 
In terms of this parameterization the action of $U$ on $A$ is then given by
\[
u(a_0+s(t')+t_0)=a_0+s(t'+[u])+t_0+\phi _u(a_0+s(t')),
\]
where $\phi_u$ is an affine-linear map from $A$ to $T_0$. The map $\phi_u$
factors over $A\to A'$ and is independent of $a_0$. So we can write 
$\phi_u(a)=\phi (a',u')$. Then the fact
that $u,v\in U$ commute implies the symmetry of $d\phi$. In particular, $\phi$
is linear in the second variable. Hence
for any $k\in \ZZ$, 
\[
u^k(a_0)=a_0+ks(u')+k\phi _u(a_0)+\tfrac{1}{2}k(k-1)d\phi ([u],[u]),
\]
where $u'\in T'$ denotes the image of $u\in U$

Suppose $u\in U-\{ 1\}$ nonzero. If $d\phi ([u],[u])=0$, then the displayed formula shows that the orbit 
$\{ u^k(a_0)|k\in \ZZ\}$  has two opposite  limiting rays (spanned by $\pm (s(u')+\phi _u(a_0))$), 
which evidently contradicts our assumption. So $d\phi ([u],[u])=0$ is nonzero, and the same formula above shows that it must belong to $C_0$. 
So if we define $\sigma $ by $\sigma (a',t'):=s(t')+\phi (a',t')$, 
then $\sigma$ has the asserted properties 
(the uniqueness of $\sigma$ is easy).
\end{proof}

We return to the face $F$ and  recall that $T_F:=V^F/V_F$.

\begin{corollary} 
$U_\G(F)$ acts trivially on $T_F$ , 
so that we can define homomorphisms of groups
\begin{align*}
j: U_\G(F)\to &\Hom (V/V^F, T_F)  \text{  such that } u(x')=x'+j_uq_F(x'),\\
k: U_\G(F)\to &\Hom(T_F,V_F)  \text{  such that } u(y)=y+k_u\pi_F(y).
\end{align*} 
with $x\in V$ and $y\in V^F$. Moreover, if $x\in C$,
then its image $x''$ in $V/V^F$ has the property that the map
$u\in U_\G(F)\mapsto j_u(x'')\in T_F$ 
is an monomorphism of groups whose image spans $T_F$.
\end{corollary}

\begin{proof} Let $x''\in C(F)$, let $A$ denote its pre-image in $V$,
and set $D=A\cap C$.
Then $D$ is a nonempty open convex subset of $A$ and by \ref{cor:Fdagger} we have 
$\Tr(D)=\bar C\cap V^F=\bar F$. 
It follows from part (ii) of Lemma \ref{lemma:stucturefaces}
that $U_\G(F)$ acts with compact fundamental set on the image of $D$ in 
$A/V_F$. 
Hence Proposition \ref{prop:unipotentgeneral} applies (with $C=\bar F$)
and we find that $U_\G(F)$ acts in $A/V_F$ faithfully as a group 
$T_F$ of translations and this group spans $T_F$.
\end{proof}

We can now complete the proof of Theorem \ref{thm:stabilizerfacemain}.

\begin{proof}[Proof of Theorem  \ref{thm:stabilizerfacemain}]
Let $s:T_F\to V^F$
be a linear section of the projection. Since $u\in
U_\G(F)$ acts trivially on $V_F$, and as $x'\mapsto
x'+j_u(x'')$ on $V/V_F$, it follows that there
exists a $\phi _u\in \Hom(V/V_F,V_F)$ such that
\[
u(x)=x+sj_u(x'')+\phi _u(x'+\tfrac{1}{2}j_u(x'')).
\]
Since $u$ acts on $V^F$ as $x\mapsto x+k_u(x')$, 
the restriction of $\phi_u$ to $T_F$ 
must be $k_u$. So $\phi_uj_u (x'')=\tfrac{1}{2}k_uj_u(x'')$. If we set 
\[
\sigma _u(x') :=sj_u(x'')+\phi _u(x'),
\]
then it also follows that the restriction of $\sigma
_u$ to $T_F$ is $k_u$. It is clear
that the map $V/V^F\rightarrow V/V_F$
induced by $\sigma _u$ is precisely $j_u$. So
$\sigma_u$ has the property $(i)$. The
assertion that $\sigma_u$ is unique for these
properties is obvious.

If $u,v\in U_\G(F)$, then
\[
v(\sigma _u(x'))=\sigma _u(x')+k_v(\sigma
_u(x')')=\sigma _u(x')+k_vj_u(x''),
\]
and so
\begin{align*}
vu(x) &=v(x+\sigma _u(x')+\tfrac{1}{2}k_uj_u(x''))\\
&=x+\sigma_v(x')+\tfrac{1}{2}k_vj_v(x')+\sigma_u
(x')+k_vj_u(x'')+\tfrac{1}{2}k_uj_u(x'')\\
&=x+(\sigma_u+\sigma_v)(x')+\tfrac{1}{2}(k_uj_u+2k_vj_u+k_vj_v)(x'').
\end{align*}
Since $vu=uv$, the symmetry property $(ii)$ follows. 
As $\sigma_{vu}$ is characterized by
\[
vu(x)=x+\sigma _{vu}(x')+\tfrac{1}{2}(k_u+k_v)(j_u+j_v)(x'').
\]
this also yields the linearity of $\sigma $. 
The same formula shows that for $r\in \ZZ$, 
\[
u^r(x)=x+r\sigma _u(x')+\tfrac{1}{2}r^2k_uj_u(x'').
\]
Property $(iii)$ follows from this.
\end{proof}

A case of interest is when $\G$ stabilizes a \emph{proper} face $F$ of $C_+$ (that is,
$\{ 0\}\subsetneq F\subsetneq C_+$). Then Theorem 
\ref{thm:stabilizerfacemain} shows that $\G$ is contains an extension of 
$\G (F)\times \G (F^\d)$ by the abelian unipotent group $U_\G (F)$ (of rank equal $\dim T_F$) as a subgroup of finite index.
If we take $F$ minimal for this property, then $\G (F)$ leaves no proper face  of $F$ invariant and
if we take $F$ maximal for this property, then $\G (F^\d)$ leaves no proper face  of $F^\d$ invariant.
In this way can often reduce our discussion to the  case when no a proper face $F$ of $C_+$ 
is preserved by $\G$. 

We can take this one a step further by reducing to the irreducible case, by which we mean that 
$\G$ does not leave invariant any proper subspace of $V$ defined over $\QQ$. 
In fact, if  $W\subset V$ is a proper $\G$-invariant subspace defined over $\QQ$, then we can distinguish three cases: 
\begin{enumerate}
\item[(a)] If $W\cap \bar C=\{ 0\}$, then the projection $\pi_W: V\to V/W$  maps $C$ onto a nondegenerate open cone and the triple $(V/W, \pi_W C,\G)$ is polyhedral.
\item[(b)] If dually, $W\cap C\not=\emptyset$, then $(W,C\cap W,\G)$ is polyhedral.
\item[(c)] If  $W$ meets $C_+$ in a proper face, then $\G$ stabilizes this face, a case we discussed above.
\end{enumerate}
Observe that in the first two cases $\G$ acts with finite kernel on $V/W$ resp.\ $W$. In case (b) this is clear, because
$W$ meets the locus where $\G$ acts properly discontinuously and  case (a) then follows by duality.

\end{document}